\documentclass[12pt,a4paper]{article}
\usepackage{amsmath}
\usepackage{amsfonts}
\usepackage{amssymb}
\usepackage{enumitem}
\usepackage{amsthm}
\usepackage{tikz}
\usepackage{xcolor}
\usepackage{graphicx, subcaption}

\usepackage[english]{babel}
\usepackage[left=1cm,right=1cm,top=2cm,bottom=2cm]{geometry}
\author{PERLA EL KETTANI$^+$, TADAHISA FUNAKI$^\ast$, 
DANIELLE HILHORST$^\%$,\\ HYUNJOON PARK$^\dagger$,
AND SUNDER SETHURAMAN$^\diamond$}
\title{Singular limit of an Allen-Cahn equation with nonlinear diffusion}

\newtheorem{thm}{Theorem}[section]
\newtheorem{lem}{Lemma}
\newtheorem{rmk}{Remark}

\def\R{{\mathbb{ R}}}
\def\e{\varepsilon}
\def\vp{\varphi}

\renewcommand{\a}{\alpha}

\newcommand{\de}{\delta}
\newcommand{\fa}{\varphi}

\newcommand{\la}{\lambda}

\newcommand{\si}{\sigma}

\def\zn{z^{(N)}}

\def\wi{{w}}
\def\wic{{w^\chi}}
\def\twi{{{w}_\infty}}
\def\twic{{{w}^{\chi^*}_\infty}}

\definecolor{darkgreen}{rgb}{0,0.5,0}

\begin{document}

\maketitle
\begin{abstract}
{
We consider an Allen-Cahn equation with nonlinear diffusion, motivated by the study
of the scaling limit of certain interacting particle systems. We investigate
its singular limit and show the generation and propagation of an interface in the limit.  
The evolution of this limit interface is governed by mean curvature flow with 
a novel, homogenized speed in terms of a surface tension-mobility parameter emerging
from the nonlinearity in our equation. 
}
\end{abstract}

\footnote{ \hskip -6.5mm
{
$^+$Aix Marseille University, Toulon Univeristy, 
Laboratory Centre de Physique Théorique, CNRS,  Marseille, France. \\
e-mail: Kettaneh.perla@hotmail.com \\
$^\ast$Department of Mathematics,
Waseda University,
3-4-1 Okubo, Shinjuku-ku,
Tokyo 169-8555, Japan. \\
e-mail: funaki@ms.u-tokyo.ac.jp \\
$^\%$CNRS and Laboratoire de Math\'ematiques, University Paris-Saclay, 
Orsay Cedex 91405,  France. \\
e-mail: Danielle.Hilhorst@math.u-psud.fr \\
$^\dagger$Department of Mathematical Sciences, 
Korean Advanced Institute of Science and Technology, 
 291 Daehak-ro, Yuseong-gu, Daejeon 34141, Korea.
e-mail: hyunjoonps@gmail.com\\
$^\diamond$Department of Mathematics,
University of Arizona,
621 N.\ Santa Rita Ave.,
Tucson, AZ 85750, USA. \\
e-mail: sethuram@math.arizona.edu
}
%Running head:  
}

\thanks{MSC 2020: 
35K57, 35B40.

\thanks{keywords: 
Allen-Cahn equation, Mean curvature flow, Singular limit, Nonlinear diffusion,
Interface, Surface tension}

\section{Introduction}

The Allen-Cahn equation with linear diffusion
\begin{align*}
	u_t = \Delta u - \dfrac{1}{\e^2} F'(u)
\end{align*}
was introduced to understand the phase separation phenomena which appears in the construction of polycrystalline materials \cite{AC1979}. Here, $u$ stands for the order parameter which describes the state of the material, $F$ is a double-well potential with two distinct local minima $\alpha_\pm$ at two different phases, and the parameter $\e > 0$ corresponds to the interface width in the phase separation process. When $\e$ is small, it is expected that $u$ converges to either of the two states $u = \alpha_+$ and $u = \alpha_-$. Thus, the limit $\e \downarrow 0$ creates a steep interface dividing two phases; this coincides with the phase separation phenomena and the limiting interface is known to evolve according to mean curvature flow; see \cite{AHM2008, Xinfu1990}.

In this paper,
we prove generation and propagation of interface properties for an Allen-Cahn equation with nondegenerate nonlinear diffusion. More precisely, we study the problem
\begin{align*}
(P^\e)~~
	\begin{cases}
	u_t
	= \Delta \vp(u)
	+ \displaystyle{ \frac{1}{\e^2}} f(u)
	&\mbox{ in } D \times \mathbb{R}^+\\
	\displaystyle{ \frac{\partial \vp(u)}{\partial \nu} }
	= 0
	&\mbox{ in } \partial D \times \mathbb{R}^+\\
	u(x,0) = u_0(x)
	&\text{ for } x \in D
	\end{cases}
	\end{align*}
where the unknown function $u$ denotes a phase function, $D$ is a smooth bounded domain in $\mathbb{R}^N, N \geq 2$, $\nu$ is the outward unit normal vector to the boundary $\partial D$ and $\e > 0$ is a small parameter. The nonlinear functions $\vp$ and $f$ satisfy the following properties.  

We assume that $f$ has exactly three zeros $f(\alpha_-) = f(\alpha_+) = f(\alpha) = 0$ where $\alpha_- < \alpha < \alpha_+$, and 
\begin{align}\label{cond_f_bistable}
	f \in C^2(\mathbb{R})
	,~
	f'(\alpha_-) < 0
	,~ f'(\alpha_+) < 0
	,~ f'(\alpha) > 0
\end{align}
so that
\begin{align}\label{cond_f_tech}
	f(s) > 0
	~
	\text{for}
	~
	s < \alpha_-
	,~
	f(s) < 0
	~
	\text{for}
	~
	s > \alpha_+.
\end{align}
We suppose that 
\begin{align}\label{cond_phi'_bounded}
	\vp \in C^4(\mathbb{R}), ~~ \vp' \geq C_\vp
\end{align}
for some positive constant $C_\vp $. We impose one more relation between $f$ and $\vp$, namely  
\begin{align}\label{cond_fphi_equipotential}
	\int_{\alpha_-}^{\alpha_+}
	\vp'(s) f(s) ds
	= 0.
\end{align}

\noindent  As for the initial condition $u_0(x)$ we assume that $u_0 \in C^2(\overline{D})$. Throughout the paper, we define $C_0$ and $C_1$ as follows:
\begin{align}
	C_0 
	&:= || u_0 || _{C^0 \left( \overline{D} \right)}
	+ || \nabla u_0 || _{C^0 \left( \overline{D} \right)}
	+ || \Delta u_0 || _{C^0 \left( \overline{D} \right)}\label{cond_C0}
	\\
	C_1 
	&:= 
	\max_{|s - \alpha| \leq I} \vp(s) +
	\max_{|s - \alpha| \leq I} \vp'(s) +
	\max_{|s - \alpha| \leq I} \vp''(s)
	,~~
	I = C_0 + \max(\alpha - \alpha_-, \alpha_+ - \alpha).\label{cond_C1}
\end{align}
Furthermore, we define $\Gamma_0$ by
\begin{align*}
	\Gamma_0 
	:= 
	\{
		x \in D: u_0(x) = \alpha
	\}.
\end{align*}
In addition, we suppose $\Gamma_0$  is a $C^{4+\nu}, 0 < \nu < 1$ hypersurface without boundary such that 
\begin{align}
	\Gamma_0 \Subset	D, \nabla u_0(x) \cdot n(x) \neq 0 \text{ if}~ x \in \Gamma_0 \label{cond_gamma0_normal} \\
		u_0 > \alpha \text{ in } D_0^+, ~~~~~~ u_0 < \alpha \text{ in } D_0^-, \label{cond_u0_inout}
\end{align}
where $D_0^-$ denotes the region enclosed by $\Gamma_0$, $D_0^+$ is the region enclosed between $\partial D$ and $\Gamma_0$, and $n$ is the outward normal vector to $D_0^-$.
It is standard that the above formulation, referred to as Problem $(P^\e)$, possesses a unique classical solution $u^\e$.

The present paper is originally motivated by the study of the scaling limit of a
Glauber+Zero-range particle system.  In this microscopic system of interacting random walks, the Zero-range part governs the rates of jumps, while the Glauber part prescribes
creation and annihilation rates of the particles.  In a companion paper \cite{EFHPS}, we show that
the system exhibits a phase separation and, under a certain space-time scaling limit, an interface
arises, in the limit macroscopic density field
of particles, evolving in time according to the motion by mean curvature.  The system is indeed well 
approximated from macroscopic viewpoint by the Allen-Cahn equation with nonlinear diffusion 
$(P^\e)$, or more precisely by its discretized equation.  Although, in this paper, we study
$(P^\e)$ under the Neumann boundary conditions, the formulation under periodic boundary conditions,
used in the particle system setting in \cite{EFHPS}, can be treated similarly; see Remark \ref{rem:1} below.

In some other physical situations, it is expected that the diffusion can depend on the order parameter
as in our case.
 In the experimental article \cite{Wagner1952}, Wagner suggested that for metal alloys the diffusion depends on the concentration. In \cite{MD2010, RLA1999}, the authors considered degenerate diffusions such as porous medium diffusions instead of linear diffusions. In \cite{FL1994}, Fife and Lacey generalized the Allen-Cahn equation, which leads them to a parameter dependent diffusion Allen-Cahn equation. Recently, \cite{FHLR2020} considered an Allen-Cahn equation with density dependent diffusion in $1$ space dimension and showed a slow motion property. However, no rigorous proof on the motion of the interface in the nonlinear diffusion context has been given for larger space dimensions $N\geq 2$.

In this context, the purpose of this article is to study the singular limit of $u^\e$ as 
 $\e \downarrow 0$.
 We first present a result on the generation of the interface.  We use the following notation:
\begin{align}\label{cond_mu_eta0}
	\mu = f'(\alpha)
	, ~~ 
	t^\e = \mu^{-1} \e^2 |\ln \e|
	, ~~
	\eta_0 = \min(\alpha - \alpha_-, \alpha_+ - \alpha).
\end{align}

\begin{thm}\label{Thm_Generation}
Let $u^\e$ be the solution of the problem $(P^\e)$, $\eta$ be an arbitrary constant satisfying $0 < \eta < \eta_0$. Then, there exist positive constants $\e_0$ and $M_0$ such that, for all $\e \in (0, \e_0)$, the following holds: 

\begin{enumerate}[label = (\roman*)]
\item for all $x \in D$
\begin{align}\label{Thm_generation_i}
	\alpha_- - \eta
	\leq
	u^\e(x,t^\e)
	\leq
	\alpha_+ + \eta;
\end{align}

\item if $u_0(x) \geq \alpha +  M_0 \e$, then
\begin{align}\label{Thm_generation_ii}
	u^\e(x,t^\e) \geq \alpha_+ - \eta;
\end{align}

\item if $u_0(x) \leq \alpha  - M_0 \e$, then
\begin{align}\label{Thm_generation_iii}
	u^\e(x,t^\e) \leq \alpha_- + \eta.
\end{align}

\end{enumerate}

\end{thm}

After the interface has been generated, the diffusion term has the same order as the reaction term. As a result the interface starts to propagate. Later, we will prove that the interface moves according to the following motion equation:
\begin{align}\label{eqn_motioneqn}
	(IP)
	\begin{cases}
		V_n = - (N - 1) \lambda_0 \kappa
		& 
		\text{ on } \Gamma_t
		\\
		\Gamma_t|_{t = 0} = \Gamma_0,
		&
		~~
	\end{cases}
\end{align}
where $\Gamma_t$ is the interface at time $t > 0$, $V_n$ is the normal velocity on the interface, $\kappa$ denotes its mean curvature, and $\lambda_0$ is a positive  constant which will be defined later (see \eqref{eqn_lambda0} and {\eqref{second lambda_0}}). It is well known that Problem $(IP)$ possesses locally in time a unique smooth solution. Fix $T > 0$ such that the solution of $(IP)$, in \eqref{eqn_motioneqn}, exists in $[0,T]$ and denote the solution by $\Gamma = \cup_{0\leq t < T} (\Gamma_t \times \{t\})$. From Proposition 2.1 of \cite{Xinfu1990} such a $T > 0$ exists, and one can deduce that $\Gamma \in C^{4 + \nu, \frac{4 + \nu}{2}}$ in $[0,T]$, given that $\Gamma_0 \in C^{4 + \nu}$. 

 The second main theorem states a result on the generation and the propagation of the interface.

\begin{thm}\label{Thm_Propagation}
	Under the conditions given in Theorem \ref{Thm_Generation}, for any given $0 < \eta < \eta_0$ there exist $\e_0 > 0$ and $C > 0$ such that 
\begin{align}\label{thm_propagation_1}
	u^\e
	\in
	\begin{cases}
		[\alpha_- - \eta, \alpha_+ + \eta]
		&
		\text{ for } x \in  D
		\\
		[\alpha_+ - \eta, \alpha_+ + \eta]
		&
		\text{ if } x \in D^+_t \setminus\mathcal{N}_{C\e}(\Gamma_t)
		\\
		[\alpha_- - \eta, \alpha_- + \eta]
		&
		\text{ if }  x \in D^-_t \setminus\mathcal{N}_{C\e}(\Gamma_t)
	\end{cases}
\end{align}
for all $\e \in (0, \e_0)$ and $t \in [t^\e,T]$, where  $D_t^-$ denotes the 
region enclosed by $\Gamma_t$, $D_t^+$ is that enclosed between $\partial D$ and 
$\Gamma_t$, and 
$\mathcal{N}_r(\Gamma_t) := \{ x \in D, dist(x, \Gamma_t) < r \}$.

\end{thm} 
This theorem implies that, after generation, the interface propagates according to the motion $(IP)$ with a width of order $\mathcal{O}(\e)$. Note that Theorems \ref{Thm_Generation} and \ref{Thm_Propagation} extend similar results for linear diffusion Allen-Cahn equations due to \cite{AHM2008}. 

We now state an approximation result inspired by a similar result proved in \cite{MH2012}.

\begin{thm}\label{thm_asymvali}
\begin{enumerate}
\item[(i)]
Let the assumptions of Theorem \ref{Thm_Propagation} hold and $\rho > 1$. Then,  the solution $u^\e$ of $(P^\e)$ satisfies
\begin{align}\label{thm_asymvali_i}
	\lim_{\e \rightarrow 0} 
	\sup_{\rho t^\e \leq t \leq T,~x \in D}
	\left|
		u^\e(x,t) 
		- U_0
		\left(
			\dfrac{d^\e(x,t)}{\e} 
		\right)
	\right|
	= 0,
\end{align}
where $U_0$ is a standing wave solution defined in \eqref{eqn_AsymptExp_U0} and $d^\e$ denotes the signed distance function associated with $\Gamma_t^\e := \{ x \in D : u^\e(x,t) = \alpha \}$, defined as follows:
\begin{align*}
	d^\e(x,t)
	=
	\begin{cases}
		dist(x,\Gamma^\e_t)
		&\text{if}~~
		x \in D^{\e,+}_t
		\\
		- dist(x,\Gamma^\e_t)
		&\text{if}~~
		x \in D^{\e,-}_t
	\end{cases}
\end{align*}
where $D^{\e,-}_t$ denotes the region enclosed by $\Gamma^\e_t$ and $D^{\e,+}_t$ denotes the region enclosed between $\partial D$ and $\Gamma^\e_t$.

\item[(ii)]
For small enough $\e > 0$ and for any $t \in [\rho t_\e, T]$, $\Gamma^\e_t$ can be expressed as a graph over $\Gamma_t$.
\end{enumerate}
\end{thm}

\begin{rmk}  \label{rem:1}
	Theorems \ref{Thm_Generation}, \ref{Thm_Propagation} and \ref{thm_asymvali} hold not only for the Neumann boundary condition of Problem $(P^\e)$ but also for periodic boundary conditions with $D = \mathbb{T}^N$, with similar proofs as given in  Sections \ref{section_3}, \ref{section_4} and \ref{section_5}.
\end{rmk}

The paper is organized as follows. In Section \ref{sec:2}, the interface motion $(IP)$ is
formally derived from the problem $(P^\e)$ as $\e \downarrow 0$.  In particular,
the constant $\lambda_0$ is obtained.  Section \ref{section_3} studies the generation
of interface and gives the proof of Theorem \ref{Thm_Generation}.
In a short time, the reaction term $f$ governs the system and the solution of 
$(P^\e)$ behaves close to that of an ordinary differential equation.
Section \ref{section_4} discusses the propagation of interface and 
Theorem \ref{Thm_Propagation} is proved.  The sub- and super-solutions are
constructed by means of two functions $U_0$ and $U_1$ formally introduced in
asymptotic expansions in Section \ref{sec:2}.
Section \ref{section_5} gives the proof of Theorem \ref{thm_asymvali}.
Finally, in the Appendix, we define the mobility $\mu_{AC}$ and the surface tension
$\si_{AC}$ of the interface, especially in our nonlinear setting, and show the relation
$\la_0= \mu_{AC}\si_{AC}$.

\section{Formal derivation of the interface motion equation}
\label{sec:2}

 In this section, we formally derive the equation of interface motion of the Problem $(P^\e)$ by applying the method of matched asymptotic expansions. To this purpose, we first define the interface $\Gamma_t$ and then derive its equation of motion. 

 Suppose that $u^\e$  converges to a step function $u$ where
\begin{align*}
	u(x,t) 
	= 
	\begin{cases}
		\alpha_+
		&
		\text{in}~ D^+_t
		\\
		\alpha_-
		&
		\text{in}~ D^-_t.
	\end{cases}
\end{align*}
Let
\begin{align*}
	\Gamma_t = \overline{D^+_t}\cap \overline{D^-_t}, \overline{D^+_t}\cup \overline{D^-_t} = D ,~t \in [0,T].
\end{align*}

 Let also $\overline{d}(x,t)$ be the signed distance function to $\Gamma_t$ defined by 
\begin{align}\label{eqn_signed_dist}
	\overline{d}(x,t)
	:=
	\begin{cases}
		- dist(x, \Gamma_t)
		& \text{ for } x \in \overline{D^-_t}
		\\
		dist(x, \Gamma_t)
		& \text{ for } x \in D^+_t.
	\end{cases}
\end{align}

 Assume that $u^\e$ has the expansions
\begin{align*}
	u^\e(x,t)
	= \alpha_\pm + \e u^\pm_1(x,t) + \e^2 u^\pm_2(x,t) + \cdots
\end{align*}
away from the interface $\Gamma$ and that

\begin{align}\label{eqn_u^eps_expansion}
	u^\e(x,t)
	= U_0(x,t,\xi)
	+ \e U_1(x,t,\xi)
	+ \e^2 U_2(x,t,\xi)
	+ \cdots
\end{align}
near $\Gamma$, where $\displaystyle{\xi = \frac{\overline{d}}{\e}}$. Here, the variable $\xi$ is given to describe the rapid transition between the regions $\{ u^\e \simeq \alpha^+ \}$ and $ \{ u^\e \simeq \alpha^- \}$. In addition, we normalize $U_0$ and $U_k$ so that
\begin{align}\label{cond_Uk_normal}
	U_0(x,t,0) = \alpha
	\nonumber \\
	U_k(x,t,0) = 0.
\end{align}
	
To match the inner and outer expansions, we require that 
\begin{align}\label{cond_U0_matching}
	U_0(x,t,\pm \infty) = \alpha_\pm, 
	~~~
	U_k(x,t,\pm \infty) = u^\pm_k(x,t)
\end{align}
for all $k \geq 2$. 

After substituting the expansion (\ref{eqn_u^eps_expansion}) into $(P^\e)$, we collect the $\e^{-2}$  terms, to obtain 
\begin{align*}
	\vp(U_0)_{zz} + f(U_0) = 0.
\end{align*}
Since this equation only depends on the variable $z$, we may assume that $U_0$ is only a function of the variable $z$, that is $U_0(x,t,z) = U_0(z)$. In view of the conditions (\ref{cond_Uk_normal}) and (\ref{cond_U0_matching}), we find that $U_0$ is the unique increasing solution of the following problem
\begin{align}\label{eqn_AsymptExp_U0}
	\begin{cases}
	(\vp(U_0))_{zz} + f(U_0) 
	= 0
	\\
	U_0(-\infty) = \alpha_-,~ U_0(0)= \alpha,~  U_0(+\infty) = \alpha_+.
	\end{cases}
\end{align}
In order to understand the nonlinearity more clearly, we set
\begin{align*}
	g(v) := f(\vp^{-1}(v)), 
\end{align*}
where $\vp^{-1}$ is the inverse function of $\vp$ and define $V_0(z) := \vp(U_0(z))$; note that such transformation is possible by the condition (\ref{cond_phi'_bounded}). Substituting $V_0$ into equation (\ref{eqn_AsymptExp_U0}) yields
\begin{align}\label{eqn_AsymptExp_V0}
	\begin{cases}
		V_{0zz} + g(V_0) = 0
		\\
		V_0(-\infty) = \vp(\alpha_-),~
		V_0(0)= \vp(\alpha), ~
		V_0(+\infty) = \vp(\alpha_+).
	\end{cases}
\end{align}
Condition (\ref{cond_fphi_equipotential}) then implies  the existence of the unique increasing solution of (\ref{eqn_AsymptExp_V0}).

Next we collect the $\e^{-1}$ terms in the asymptotic expansion. In view of the definition of $U_0(z)$ and the condition (\ref{cond_Uk_normal}), we obtain the following problem
\begin{align}\label{eqn_AsymptExp_U1}
	\begin{cases}
	(\vp'(U_0) \overline{U_1})_{zz} + f'(U_0)\overline{U_1} 
	= \overline{d}_t U_{0z} - (\vp(U_0))_z \Delta \overline{d}
	\\
	\overline{U_1}(x,t,0) = 0, ~~~ \vp'(U_0) \overline{U_1} \in L^\infty(\mathbb{R}).
	\end{cases}
\end{align}
To prove the existence of solution to (\ref{eqn_AsymptExp_U1}), we consider the transformed function $\overline{V_1} = \vp'(U_0)\overline{U_1}$, which gives the problem
\begin{align}\label{eqn_AsymptExp_V1}
	\begin{cases}
		\overline{V_{1}}_{zz} + g'(V_0)\overline{V_1} 
		= 
		\displaystyle{\frac{V_{0z}}{\vp'(\vp^{-1} (V_0) )}} \overline{d}_t
		-  
		V_{0z}	\Delta \overline{d}
		\\
		\overline{V_1}(x,t,0) = 0, ~~~
		\overline{V_1} \in L^\infty(\mathbb{R}).
	\end{cases}
\end{align}
Now, Lemma 2.2 of \cite{AHM2008} implies the existence and uniqueness of $V_1$ provided that  
\begin{align*}
	\int_\R 
	\left(
		\frac{1}{\vp'(\vp^{-1}(V_0))}
		\overline{d}_t
		- \Delta \overline{d}
	\right)
	V_{0z}^2
	= 0.
\end{align*}
Substituting $V_0 = \vp(U_0)$ and $ V_{0z} = \vp'(U_0) U_{0z} $ yields 
\begin{align}
	\overline{d}_t
	= \frac
	{\int_\R V_{0z}^2}
	{\int_\R \frac{V_{0z}^2}{\vp'(\vp^{-1}(V_0))}}
	\Delta \overline{d}
	= \frac
	{\int_\R (\vp'(U_0) U_{0z})^2}{\int_\R \vp'(U_0) U_{0z}^2}
	\Delta \overline{d}.
\end{align}
It is known that $\overline{d}_t=-V_n$, where $V_n$ is equal to the normal velocity on the interface $\Gamma_t$, and $\Delta \overline{d}$ is equal to $(N - 1) \kappa$, where $\kappa$ is the mean curvature of $\Gamma_t$. Thus, we obtain the equation of motion of the interface $\Gamma_t$,

\begin{align}
	V_n = -( N - 1 ) \lambda_0 \kappa,
\end{align}
where 
\begin{align}
	\lambda_0 
	=
	\frac
	{\int_\R (\vp'(U_0) U_{0z})^2}{\int_\R \vp'(U_0) U_{0z}^2}.  \label{eqn_lambda0}
\end{align}
The constant $\lambda_0$ is interpreted as the surface tension $\si_{AC}$ multiplied by the 
mobility $\mu_{AC}$ of the interface; see Appendix.  In particular, $(IP)$ coincides with 
the equation (1) in \cite{AC1979}.

Finally, we derive an explicit form of $\lambda_0$. Indeed, we multiply
the equation \eqref{eqn_AsymptExp_U0} by ${\vp(U_0)}_z$, yielding
\begin{equation*}
{\vp(U_0)}_{zz} {\vp(U_0)}_z+f(U_0) {\vp(U_0)}_{z}=0\,.
\end{equation*}
Integrating from $-\infty$ to $z$, we obtain
$$
\frac 12 \big[\vp(U_0)_z\big]^2(z)+\int _{-\infty} ^{z} f(U_0) {\vp(U_0)}_{z} dz=0\,
$$
or alternatively 
$$
\frac 12 \big[\vp(U_0)_z\big]^2(z)+\int _{\alpha_-} ^{U_0(z)} f(s) \vp'(s) ds=0\,.
$$
Hence, 
\begin{equation}\label{intrinsic}
{\vp(U_0)_z}(z)= \sqrt 2 \sqrt {W(U_0(z))}\,,
\end{equation}
where $W$ is given by 
\begin{align}\label{eq:28}
	W(u) 
	= - \int ^{u} _{\alpha_-} f(s) \vp'(s) ds 
	= \int _{u} ^{\alpha_+} f(s) \vp'(s) ds,
\end{align}
the last equality holdinf by \eqref{cond_fphi_equipotential}. It follows that
$$
\int_{\R} {\vp(U_0)}_z{U_{0z}}(z) dz= \sqrt 2 \int_{\R} \sqrt{W(U_0(z))} U_{0z}(z) dz
$$
so that also
$$
\int _{\R} {\vp'(U_0)} U_{0z}^2(z)dz= \sqrt 2 \int _{\alpha_-} ^{\alpha_+}  \sqrt{W(u)}du.
$$
Similarly, since
$$
\int_\R (\vp'(U_0) U_{0z})^2dz =  \sqrt 2 \int_\R (\vp'(U_0) \sqrt{W(U_0(z))}  U_{0z}) dz,
$$
we get 
\begin{align}\label{eq:29}
	\int_\R (\vp'(U_0) U_{0z})^2dz =  \sqrt 2 \int _{\alpha_-} ^{\alpha_+} \vp'(u) \sqrt{W(u)} du,
\end{align}
so that we finally obtain the formula 
\begin{align}
\label{second lambda_0}
\lambda_0 = \frac
{\int _{\alpha_-} ^{\alpha_+} \vp'(u) \sqrt{W(u)} du}{\int _{\alpha_-} ^{\alpha_+}  \sqrt{W(u)}du}.
\end{align}
Note that if $\vp(u)=u$, the case of the linear diffusion Allen-Cahn equation, we recover the value $\lambda_0 =1$ as expected.

\section{Generation of the interface}\label{section_3}
 In this section, we prove Theorem \ref{Thm_Generation} on the generation of the interface.
  The main idea, based on the comparison principle Lemma \ref{lem_comparison}, is to construct suitable sub- and super-solutions.  The proof of Theorem \ref{Thm_Generation} is given in Section \ref{proof_subsec_thm_gen}. 
  
\subsection{Comparison principle}

\begin{lem}\label{lem_comparison}
	Let $v \in C^{2,1} (\overline{D} \times \R^+)$ satisfy
	\begin{align*}
	(P)~
	\begin{cases}
	v_t
	\geq \Delta \vp(v)
	+ \displaystyle{ \frac{1}{\e^2}} f(v)
	&\mbox{ in } D \times \mathbb{R}^+\\
	{ \dfrac{\partial \vp(v)}{\partial \nu} }
	= 0
	&\mbox{ in } \partial D \times \mathbb{R}^+\\
	v(x,0) \geq u_0(x)
	&\text{ for } x \in D.
	\end{cases}
\end{align*}
Then, $v$ is a super-solution of Problem $(P^\e)$ and we have
\begin{align*}
	v(x,t) \geq u^\e(x,t)
	,~~
	(x,t) \in D \times \R^+.
\end{align*}
If $v$ satisfies the opposite inequalities in Problem $(P)$, then $v$ is a sub-solution of Problem $(P^\e)$ and we have
\begin{align*}
	v(x,t) \leq u^\e(x,t)
	,~~
	(x,t) \in D \times \R^+.
\end{align*}

\end{lem}

\begin{proof}
{Consider the inequality satisfied for the difference of
a super-solution $v$ and a solution $u^\e$.  Apply the maximum principle to
the function $w : = v - u^\e$ to see that it is positive.}
\end{proof}

\subsection{Solution of the corresponding ordinary differential equation}

  In the first stage of development, we expect that the solution behaves as that of the corresponding ordinary differential equation:

\begin{align}\label{eqn_generation_ODE}
	\begin{cases}
		Y_\tau(\tau, \zeta) = f(Y(\tau,\zeta)) & \tau > 0\\
		Y(0,\zeta) = \zeta & \zeta \in \mathbb{R}.
	\end{cases} 
\end{align}
We deduce the following result from \cite{AHM2008}.

\begin{lem}\label{Lem_Generation_Matthieu}
	Let $\eta \in (0, \eta_0)$ be arbitrary. Then, there exists a positive constant $C_Y 
	= C_Y(\eta)$ such that the following holds:
\begin{enumerate}[label =(\roman*)]
\item There exists a positive constant $\overline{\mu}$ such that for all $\tau > 0$ and all $\zeta \in (-2C_0, 2C_0)$,
\begin{align}\label{lem_gen_1}
	e^{- \overline{\mu} \tau} 
	\leq 
	Y_\zeta(\tau,\zeta) 
	\leq 
	C_Y e^{\mu \tau}.
\end{align}

\item For all $\tau > 0$ and all $\zeta \in (-2C_0, 2C_0)$,
$$
	\left|
	\frac{Y_{\zeta \zeta}(\tau, \zeta)}{Y_\zeta(\tau, \zeta)} 
	\right|
	\leq C_Y (e^{\mu \tau} - 1).
$$

\item There exists a positive constants $\e_0$ such that, for all $\e \in (0, \e_0)$, we have
\begin{enumerate}
\item for all $\zeta \in (-2C_0, 2C_0)$
\begin{align}\label{Lem_Generation_i}
	\alpha_- - \eta
	\leq
	Y(\mu^{-1} |\ln \e|, \zeta)
	\leq
	\alpha_+ + \eta;
\end{align}

\item if $\zeta \geq \alpha + C_Y \e$, then
\begin{align}\label{Lem_Generation_ii}
	Y(\mu^{-1} |\ln \e|, \zeta) \geq \alpha_+ - \eta;
\end{align}

\item if $\zeta \leq \alpha - C_Y \e$, then
\begin{align*}
	Y(\mu^{-1} |\ln \e|, \zeta) \leq \alpha_- + \eta.
\end{align*}
\end{enumerate}

\end{enumerate}

\end{lem}
\begin{proof}
These results can be found in Lemma 4.7 and Lemma 3.7 of \cite{AHM2008}, except for \eqref{lem_gen_1}. To prove \eqref{lem_gen_1}, 
we follow similar computations as in Lemma 3.2 of \cite{AHM2008}.
Differentiating \eqref{eqn_generation_ODE} by $\zeta$, we obtain 
\begin{align*}
	\begin{cases}
		Y_{\zeta \tau}(\tau, \zeta) = f'(Y(\tau,\zeta))Y_\zeta, & \tau > 0\\
		Y_\zeta(0,\zeta) = 1, & ~
	\end{cases}
\end{align*}
which yields the following equality,
\begin{align}\label{lem_gen_2}
	Y_\zeta(\tau, \zeta)
	=
	\exp
	\left[
		\int_0^\tau f'(Y(s,\zeta))
	\right].
\end{align}
Hence, for $\zeta = \alpha$,
\begin{align*}
	Y_\zeta(\tau, \alpha)
	=
	\exp
	\left[
		\int_0^\tau f'(Y(s,\alpha))
	\right]	
	=
	e^
	{
		 \mu \tau
	},
\end{align*}
where the last equality follows since $Y(\tau,\alpha) = \alpha$. Also, for $\zeta = \alpha_\pm$, by \eqref{cond_f_bistable}, we have
\begin{align*}
	Y_\zeta(\tau, \alpha_\pm)
	\leq 
	e^
	{
		 \mu \tau
	}.
\end{align*}
 For $\zeta \in (\alpha_- + \eta, \alpha_+ - \eta) \setminus \{\alpha\}$, Lemma 3.4 of \cite{AHM2008} guarantees the upper bound of $Y_\zeta$ in \eqref{lem_gen_1}. 
We only need to consider the case that $\zeta \in (-2C_0, 2C_0 ) \setminus (\alpha_- + \eta, \alpha_+ - \eta)$. It follows from \eqref{cond_f_bistable} that we can choose a positive constant $\eta$ and $\overline{\eta}$ such that 
\begin{align}\label{lem_gen_4}
	f'(s) < 0
	~, 
	s \in I,
\end{align}
where $I :=  (\alpha_- - \overline{\eta}, \alpha_- + \eta) \cup (\alpha_+ - \eta, \alpha_+ + \overline{\eta})$. Moreover, \eqref{cond_f_bistable} and \eqref{cond_f_tech} imply 
\begin{align}\label{lem_gen_5}
	Y(\tau,\zeta) 
	\in 
	J,
\end{align} 
for $\zeta 	\in J$ where $J := (\min\{- 2C_0, \alpha_- - \overline{\eta}\}, \alpha_- + \eta) 
	\cup 
	(\alpha_+ - \eta, \max\{ 2C_0, \alpha_+ + \overline{\eta}\})$.
 Thus, \eqref{lem_gen_2}, 
 \eqref{lem_gen_4} and \eqref{lem_gen_5} guarantee the upper bound of \eqref{lem_gen_1} for $\zeta \in I$, which leaves us only the case $\zeta \in (- 2C_0, 2C_0) \setminus I.$ 
 
 We consider now the case $\zeta \in (\alpha_+ + \overline{\eta}, 2 C_0)$; the case of $\zeta \in (-2C_0, \alpha_- - \overline{\eta})$ can be analysed in a similar way. By (3.13) in \cite{AHM2008}, we have
\begin{align}\label{lem_gen_3}
	\ln Y_\zeta(\tau, \zeta)
	=
	f'(\alpha_+) \tau + \int_\zeta^{Y(\tau,\zeta)} \tilde{f}(s) ds
	,~ {\rm \ and \ }
	\tilde{f}(s)
	= \dfrac{f'(s) - f'(\alpha_+)}{f(s)}.
\end{align}
Note that $\tilde{f}(s) \to \dfrac{f''(\alpha_+)}{f'(\alpha_+)}$ as $s \to \alpha_+$, so that
$\tilde{f}$ may be extended as a continuous function. We define
\begin{align*}
	\tilde{F} 
	:=
	\Vert \tilde{f} \Vert_{L^{\infty} (\alpha_+, \max\{ 2C_0, \alpha_+ + \overline{\eta}\})}.
\end{align*}
Since \eqref{cond_f_tech} yields $Y(\tau, \zeta) > \alpha_+$ for $\zeta \in (\alpha_+ + \overline{\eta}, 2 C_0)$, by \eqref{lem_gen_3} we can find a constant $C_Y$ large enough such that 
\begin{align*}
	Y_\zeta(\tau,\zeta)
	\leq 
	C_Y e^{f'(\alpha_+) \tau}
	\leq C_Y e^{\mu \tau}.
\end{align*}
Thus, we obtain the upper bound of \eqref{lem_gen_1}.
\newline
For the lower bound, we first define 
\begin{align*}
	\overline{\mu} := - \min_{s \in I'} f'(s),~ I' = [- 2C_0, 2C_0] \cup [\alpha_-, \alpha_+].
\end{align*}
Note that $\overline{\mu} > 0$ by \eqref{cond_f_bistable}. Thus, by \eqref{lem_gen_2}, we obtain
\begin{align*}
	Y_\zeta(\tau, \zeta) \geq e^{- \overline{\mu} \tau}.
\end{align*}
\end{proof}

\subsection{Construction of sub- and super-solutions}

 We now construct sub- and super-solutions for the proof of Theorem \ref{Thm_Generation}. For simplicity, we first consider the case where 
\begin{align}\label{cond_subsuper_Neumann}
	\frac{\partial u_0}{\partial \nu} = 0 \text{ on } \partial D.
\end{align}
In this case, we define sub- and super- solution as follows:
\begin{equation*}
	w^{\pm}_\e(x,t)
	= Y 
	\left(
		\frac{t}{\e^2},
		u_0(x) 
		\pm 
		\e^2 C_2
		\left( 
			e^{\mu t/\e^2} - 1
		\right)
	\right)\\
		= Y 
	\left(
		\frac{t}{\e^2},
		u_0(x) 
		\pm 
		P(t)
	\right)
\end{equation*}
for some the constant $C_2$. 
In the general case, where (\ref{cond_subsuper_Neumann}) does not necessarily hold, we need to modify $w^{\pm}_\e$ near the boundary $\partial D$. This will be discussed later in the proof of Theorem \ref{Thm_Generation}; see after equation \eqref{eqn_proofofgeneration}.

\begin{lem}\label{Lem_generation_with_homo_Neumann}
Assume (\ref{cond_subsuper_Neumann}). Then, there exist positive constants $\e_0$ and $C_2, \overline{C}_2$ independent of $\e$ such that, for all $\e \in (0,\e_0)$, $w^{\pm}_\e$ satisfies
\begin{align}\label{eqn_Gen_subsuper}
	\begin{cases}
		\mathcal{L} (w^-_\e) < - \overline{C}_2 e^{-\frac{\overline{\mu} t}{\e^2}} < \overline{C}_2 e^{-\frac{\overline{\mu} t}{\e^2}} < \mathcal{L}(w^+_\e)
		&
		\text{ in } \overline{D} \times [0,t^\e]
		\\
		\displaystyle{\frac{\partial w^-_\e}{\partial \nu}
		= \frac{\partial w^+_\e}{\partial \nu}}
		= 0
		&
		\text{ on } \partial D \times [0,t^\e].
	\end{cases}
\end{align}

\end{lem}

\begin{proof}

We only prove that $w^{+}_\e$ is the desired super-solution; the case for $w^-_\e$ can be treated in a similar way. The assumption (\ref{cond_subsuper_Neumann}) implies 
$$
	\frac{\partial w^\pm_\e}{\partial \nu} = 0 \text{ on } \partial D \times \mathbb{R}^+
$$
	
Define the operator $\mathcal{L}$ by
$$
	\mathcal{L} u = u_t - \Delta \vp(u) - \frac{1}{\e^2} f(u).
$$
Then, direct computation with $\tau = t/\e^2$ gives
\begin{align*}
	\mathcal{L}(w^+_\e) 
	&= \frac{1}{\e^2} Y_\tau 
	+ P'(t) Y_\zeta
	- \left(
	\vp''(w^+_\e) | \nabla u_0|^2 (Y_\zeta)^2
	+ \vp'(w^+_\e) \Delta u_0 Y_\zeta
	+ \vp'(w^+_\e) |\nabla u_0|^2 Y_{\zeta\zeta}
	+ \frac{1}{\e^2} f(Y)
	\right)\\
	&= \frac{1}{\e^2} ( Y_\tau - f(Y))
	+ Y_{\zeta}
	\left(
	P'(t) 
	- \left(
	\vp''(w^+_\e) | \nabla u_0|^2 Y_\zeta
	+ \vp'(w^+_\e) \Delta u_0
	+ \vp'(w^+_\e) |\nabla u_0|^2 \frac{Y_{\zeta\zeta}}{Y_\zeta}
	\right)
	\right).
\end{align*}	
By the definition of $Y$, the first term on the right-hand-side vanishes. By choosing $\e_0$ sufficiently small, for $0 \leq t \leq t_\e$, we have 
$$
	P(t) 
	\leq P(t^\e)
	= \e^2 C_2(e^{\mu t^\e/\e^2} - 1) 	
	= \e^2 C_2(\e^{-1} - 1) < C_0.
$$
Hence, $|u_0 + P(t)| < 2C_0$. Applying Lemma \ref{Lem_Generation_Matthieu}, \eqref{cond_C0} and \eqref{cond_C1} gives 
\begin{align*}
	\mathcal{L} w^+_\e
	&\geq
	Y_\zeta \left(
	C_2 \mu e^{\mu t /\e^2}
	- (
	C_0^2 C_1 C_Y e^{\mu t / \e^2}
	+ C_0 C_1
	+ C_0^2 C_1 C_Y (e^{\mu t / \e^2} - 1))
	\right)\\
	&=
	Y_\zeta \left(
	(C_2 \mu - C_0^2 C_1 C_Y - C_0^2 C_1 C_Y)e^{\mu t / \e^2}
	+ C_0^2 C_1C_Y 
	- C_0 C_1
	\right).
\end{align*}
By \eqref{lem_gen_1},  for $C_2$ large enough, we can find a positive constant $\overline{C}_2$ independent to $\e$ such that 
$$
	\mathcal{L} w^+_\e \geq \overline{C}_2 e^{-\frac{\overline{\mu}  t}{\e^2}}.
$$
Thus, $w^+_\e$ is a super-solution for Problem $(P^\e)$.
\end{proof}

\subsection{Proof of Theorem \ref{Thm_Generation}}
\label{proof_subsec_thm_gen}

We deduce from the comparison principle Lemma \ref{lem_comparison} and the construction of the sub- and super-solutions that 
\begin{align}\label{eqn_proofofgeneration}
	w^-_\e(x,t^\e)
	\leq
	u^\e(x,t^\e)
	\leq
	w^+_\e(x,t^\e)
\end{align}
under the condition (\ref{cond_subsuper_Neumann}).

 If \eqref{cond_subsuper_Neumann} does not hold, one can modify the functions $w^\pm$ as follows:  
from condition (\ref{cond_u0_inout}), there exist positive constants $d_0$ and $\rho$ such that 
(i) the distance function 
$d(x,\partial D) $ is smooth enough on $\{ x \in D : d(x,\partial D) < 2 d_0 \}$ and
(ii) $u_0(x) \geq \alpha + \rho$ if $d(x, \partial D) \leq d_0$. 
Let $\xi$ be a smooth cut-off function defined on $[0,+\infty)$ such that $0 \leq \xi \leq 1, \xi(0) = \xi'(0) = 0$ and $\xi(z) = 1$ for $z \geq d_0$. Define
\begin{align*}
	u_0^+
	&:= \xi(d(x,\partial D)) u_0(x) 
	+\left[
		1 - \xi(d(x, \partial D)) 
	\right]
	\max_{\overline{D}} u_0
	\\
	u_0^-
	&:= \xi(d(x,\partial D)) u_0(x) 
	+\left[
		1 - \xi(d(x, \partial D)) 
	\right]
	(\alpha + \rho).
\end{align*}
Then, $u_0^- \leq u_0 \leq u_0^+$ and $u_0^\pm$ satisfy the homogeneous Neumann boundary condition  \eqref{cond_subsuper_Neumann}. Thus, by using a similar argument as in the proof of Lemma \ref{Lem_generation_with_homo_Neumann}, we may find sub- and super-solutions as follows,
\begin{align*}
	w^{\pm}_\e(x,t)
	= Y 
	\left(
		\frac{t}{\e^2},
		u_0^\pm(x) 
		\pm 
		\e^2 C_2
		\left( 
			e^{\mu t/\e^2} - 1
		\right)
	\right).
\end{align*}

We now show \eqref{Thm_generation_i}, \eqref{Thm_generation_ii} and \eqref{Thm_generation_iii}.  By the definition of $C_0$ in (\ref{cond_C0}), we have
\begin{align*}
	-C_0
	\leq \min_{x \in \overline{D}} u_0(x)
	<
	\alpha + \rho.
\end{align*}
Thus, for $\e_0$ small enough, we have that 
$$
	- 2 C_0
	\leq
	u^\pm_0(x) \pm (C_2 \e - C_2 \e^2)
	\leq
	2 C_0
	~~~
	\text{ for }
	x \in D
$$
holds for any $\e \in (0, \e_0)$.
Thus, the assertion (\ref{Thm_generation_i}) is a direct consequence of (\ref{Lem_Generation_i}) and (\ref{eqn_proofofgeneration}).

For (\ref{Thm_generation_ii}), first we choose $M_0$ large enough so that 
$M_0 \e - C_2 \e + C_2 \e^2 \geq C_Y \e$. Then, for any $x \in D$ such that $u^-_0(x) \geq \alpha + M_0 \e$, we have

$$
	u_0^-(x) -
		\e^2 C_2
		\left( 
			e^{\mu t/\e^2} - 1 \right)\geq u^-_0(x) - (C_2 \e - C_2 \e^2)
	\geq 
	\alpha + M_0 \e - C_2 \e + C_2 \e^2
	\geq 
	\alpha + C_Y \e.
$$
	Therefore, with (\ref{Lem_Generation_ii}) and (\ref{eqn_proofofgeneration}), we see that 

$$
	u^\e(x,t^\e) \geq \alpha_+ - \eta
$$
	
\noindent for any $x \in D$ such that $u^-_0(x) \geq \alpha + M_0 \e$, which implies (\ref{Thm_generation_ii}). 
Note that (\ref{Thm_generation_iii}) can be shown in the same way. This completes the proof of Theorem \ref{Thm_Generation}. \hfill\qed

\section{Propagation of the interface}\label{section_4}

The main idea of the proof of Theorem \ref{Thm_Propagation} is that we proceed by imbrication:
By the comparison principle Lemma \ref{lem_comparison}, we show at the generation time that $u^+(x,0) \geq w^+(x, t^\varepsilon)$ and
$u^-(x,0) \leq w^-(x,t^\varepsilon)$ so that we can pass continuously from the generation
of interface sub- and super-solutions to the propagation of interface sub- and super-solutions.

 To this end, we first introduce a modified signed distance function, and several estimates 
on the functions $U_0$ and $U_1$
useful in the sub- and super-solution construction, before showing Theorem \ref{Thm_Propagation} in Section \ref{proof_thm_prop}.

\subsection{A modified signed distance function}

 We introduce a useful cut off signed distance function $d$ as follows. Recall the signed distance function $\overline{d}$ defined in \eqref{eqn_signed_dist}, and interface $\Gamma_t$ satisfying \eqref{eqn_motioneqn}. Choose $d_0 > 0$ small enough so that the signed distance function $\overline{d}$ is smooth in the set  
$$
\{
(x,t) \in \overline{D} \times [0,T] , | \overline{d}(x,t) | < 3 d_0
\}
$$
and that 
$$
	dist(\Gamma_t, \partial D) \geq 3 d_0
	\text{ for all } t \in [0,T].
$$
Let $h(s)$ be a smooth { non-decreasing} function on $\mathbb{R}$ such that 
$$
	h(s) = 
	\begin{cases}
		s & \text{if}~ |s| \leq d_0\\
		-2d_0 & \text{if}~ s \leq -2d_0\\
		2d_0 & \text{if}~ s \geq 2d_0.
	\end{cases}
$$
We then define the cut-off signed distance function $d$ by 
$$
	d(x,t) = h(\overline{d}(x,t)), ~~~ (x,t) \in \overline{D} \times [0,T].
$$
Note, as $d$ coincides with $\overline{d}$ in the region
$$
	\{ 
	(x,t) \in D \times [0,T] : | d(x,t)| < d_0
	\},
$$
that we have 
\begin{align*}
	d_t 
	= \lambda_0 \Delta d
	~\text{ on }~ \Gamma_t.
\end{align*}
Moreover, $d$ is constant near $\partial D$ and the following properties hold.
\begin{lem}\label{Lem_d_bound}
	There exists a constant $C_d > 0$ such that 
\begin{enumerate}[label = (\roman*)]
	\item 
	$|d_t| + |\nabla d| + |\Delta d| \leq C_d$,
		
	\item 
	$
	\left|
	d_t - \lambda_0 \Delta d
	\right|
	\leq 
	C_d |d|
	$
\end{enumerate}
in $\overline{D} \times [0,T]$.
\end{lem}

\subsection{Estimates for the functions $U_0, U_1$}

Here, we give estimates for the functions which will be used to construct the sub- and super-solutions. Recall that $U_0$ (cf. \eqref{eqn_AsymptExp_U0}) is a solution of the equation 
\begin{align*}
	(\vp(U_0))_{zz} + f(U_0) = 0.
\end{align*}
	
We have the following lemma.

\begin{lem}\label{Lem_U0_bound}
 There exists constants $\hat{C}_0, \lambda_1 > 0$ such that for all $z\in \mathbb{R}$,
\begin{enumerate}[label = (\roman*)]
	\item
	$
	|U_0| , ~ |U_{0z}| , ~ |U_{0zz}| 
	\leq \hat{C}_0,
	$
	
	\item
	$
	|U_{0z}|, ~ |U_{0zz}|
	\leq \hat{C}_0 \exp(- \lambda_1 |z|).
	$
	
\end{enumerate}
\end{lem}

\begin{proof}
Recall that $V_0 = \vp(U_0)$ satisfies the equation (\ref{eqn_AsymptExp_V0}) with $\vp \in C^4(\R)$. Lemma 2.1 of \cite{AHM2008} implies that there exist some positive constants $\overline{C}_0$ and $\lambda_1$ such that, for all $z\in \mathbb{R}$,
\begin{align*}
	&|V_0| ,~ |V_{0z}| ,~ |V_{0zz}| 
	\leq \overline{C}_0;
	\\
	&|V_{0z}| ,~ |V_{0zz}|
	\leq \overline{C}_0 \exp(- \lambda_1 |z|),
\end{align*}
and therefore similar bounds for $U_0$.
\end{proof}

In terms of the cut-off signed distance 
function $d=d(x,t)$, for each $(x,t) \in \overline{D}\times [0,T]$,
we define $U_1(x,t,\cdot) : \R \rightarrow \R$ }
as the solution of the following equation:
\begin{align}\label{eqn_U1_bar}
	\begin{cases}
		(\vp'(U_0) U_1)_{zz} + f'(U_0)U_1 
		= (\lambda_0 U_{0z} - (\vp(U_0))_z) \Delta d\\
		U_1(x,t,0) = 0, ~~~
		\vp'(U_0) U_1 \in L^\infty(\mathbb{R}).
	\end{cases}
\end{align}
Existence of the solution $U_1$ can be shown in the same way as that for
$\overline{U_1}$ in \eqref{eqn_AsymptExp_U1}. Finally, we give the following estimates for  $U_1=U_1(x,t,z)$.

\begin{lem}\label{Lem_U1_bound}
 There exists a constant $\hat{C}_1, \lambda_1 > 0$ such that for all $z \in \mathbb{R}$ 
\begin{enumerate}[label = (\roman*)]
	\item
	$
	|U_1| ,~ |{U_1}_z| ,~ |{U_1}_{zz}| ,~ |\nabla {U_1}_z| ,~ |\nabla {U_1}| ,~ |\Delta{U_1}| ,~ |U_{1t}| \leq \hat{C}_1,
	$
	\item
	$
	|{U_1}_z| ,~ |{U_1}_{zz}|,~ |\nabla {U_1}_z| \leq \hat{C}_1 \exp(- \lambda_1 |z|).
	$
\end{enumerate}
{
Here, the operators $\nabla$ and $\Delta$ act on the variable $x$.}
\end{lem}

\begin{proof}
Define $V_1(z) := \vp'(U_0(z)) {U}_1(z)$. As in  (\ref{eqn_AsymptExp_U1}), we obtain an equation for $V_1$:
\begin{align}\label{eqn_AsymptExp_V1}
	\begin{cases}
		V_{1zz} + g'(V_0)V_1 
		= 
		\Big[
		\lambda_0 \displaystyle{\frac{V_{0z}}{\vp'(\vp^{-1} (V_0) )} }
		-  V_{0z}
		\Big]
		\Delta d
		\\
		V_1(x,t,0) = 0, ~~~
		V_1 \in L^\infty(\mathbb{R}).
	\end{cases}
\end{align}
Applying Lemmas 2.2 and  2.3 of \cite{AHM2008} to (\ref{eqn_AsymptExp_V1}) implies the boundedness of $V_1, V_{1z}, V_{1zz}$. Moreover, since $d$ is smooth in $\overline{D} \times [0,T]$, we can apply Lemma 2.2 of \cite{AHM2008} to obtain the boundedness of $\nabla V_1, \Delta V_1$. The desired estimates for the function $U_1$ now follows via the smoothness of $\varphi$ as in the proof of Lemma \ref{Lem_U0_bound}.
\end{proof}

\subsection{Construction of sub- and super-solutions}

We construct candidates sub- and super-solutions as follows:  Given $\e > 0$, define  
\begin{align}
\label{star0}
	u^\pm(x,t)
	= U_0
	\left(
		\frac{d(x,t) \pm \e p(t)}{\e}
	\right)
	+ \e U_1 
	\left(x,t,
		\frac{d(x,t) \pm \e p(t)}{\e}
	\right)
	\pm q(t)
\end{align}
where 
\begin{align*}
	&	p(t) = - e^{- \beta t/ \e^2} + e^{Lt} + K,\\
	&q(t) = \sigma
	\left(
	\beta e^{- \beta t / \e^2} + \e^2 L e^{Lt}
	\right),
\end{align*}
in terms of positive constants $\e, \beta, \sigma, L, K$. Next, we give specific conditions for these constants which will be used to show that indeed $u^\pm$ are sub- and super-solutions.
We assume that the positive constant $\e_0$ obeys
\begin{align}\label{eqn_cond_elc}
	\e_0^2 L e^{LT} \leq 1, ~~~ 
	\e_0\hat{C}_1 \leq \frac{1}{2}.
\end{align}
	
We first give a result on the boundedness of  $f'(U_0(z)) + (\vp'(U_0(z))_{zz}$.

\begin{lem}\label{Lem_f'+phi'_bound}
	There exists $b > 0$ such that $f'(U_0(z)) + (\vp'(U_0))_{zz} < 0$ on $\{ z : U_0(z) \in [\alpha_-,~\alpha_- + b] \cup [\alpha_+ - b ,~\alpha_+]\}$.
\end{lem}

\begin{proof}
We can choose $b_1, \mathcal{F} > 0$ such that 
\begin{align*}
	f'(U_0(z)) < - \mathcal{F}
\end{align*}
on $\{ z : U_0(z) \in [\alpha_-,~\alpha_- + b_1] \cup [\alpha_+ - b_1 ,~\alpha_+]\}$.

 Note that $(\vp'(U_0))_{zz} = \vp'''(U_0) U^2_{0z} + \vp''(U_0) U_{0zz}$. From Lemma \ref{Lem_U0_bound}, we can choose $b_2 > 0$ small enough so that 
\begin{align*}
	| (\vp'(U_0))_z | < \mathcal{F},~~~| (\vp'(U_0))_{zz} | < \mathcal{F}
\end{align*}
on $\{ z : U_0(z) \in [\alpha_-,~\alpha_- + b_2] \cup [\alpha_+ - b_2 ,~\alpha_+]\}$. Define $b := \min \{b_1, b_2 \}$. Then, we have
\begin{align*}
	f'(U_0(z)) + (\vp'(U_0))_{zz}
	<\mathcal{F} - \mathcal{F}
	= 0.
\end{align*}
\end{proof}

 Fix $b > 0$ which satisfies the result of Lemma \ref{Lem_f'+phi'_bound}. Denote $ J_1 := \{ z : U_0(z) \in [\alpha_-,~\alpha_- + b] \cup [\alpha_+ - b ,~\alpha_+]\}, J_2 = \{ z : U_0(z) \in [\alpha_- + b,~\alpha_+ - b]\}$. Let 
\begin{align}\label{cond_beta}
	\beta 
	:= - \sup
	\left\{
		\frac{f'(U_0(z)) + (\vp'(U_0(z)))_{zz}}{3} : z \in J_1
	\right\}.
\end{align}
The following result plays an important role in verifying sub- and super-solution properties.

\begin{lem}\label{Lem_E3_bound}
	There exists a constant $\sigma_0$ small enough such that for every $0 < \sigma < \sigma_0$, we have 
	$$
		U_{0z} - \sigma (f'(U_0) + (\vp'(U_0))_{zz}) \geq 3 \sigma \beta.
	$$
\end{lem}
\begin{proof}

To show the assertion, it is sufficient to show that there exists $\sigma_0$ such that, for all $0 < \sigma < \sigma_0$,
\begin{align}\label{lem_E3_1}
	\frac{U_{0z}}{\sigma} 
	- \left( f'(U_0) + (\vp'(U_0))_{zz} \right)
	\geq 3 \beta.
\end{align}
We prove the result on each of the sets $J_1, J_2$.

On the set $J_1$,
note that  $U_{0z} > 0$ on $\mathbb{R}$. If $z \in J_1$, for any $\sigma > 0$ we have
$$
	\frac{U_{0z}}{\sigma} 
	- \left( f'(U_0) + (\vp'(U_0))_{zz} \right)
	> - \sup_{z \in J_1} ( f'(U_0) + (\vp'(U_0))_{zz} )
	= 3  \beta.
$$	

On the set $J_2$, which is compact
 in $\mathbb{R}$, there exists positive constants $c_1, c_2$ such that
\begin{align*}
	U_{0z} \geq c_1
	,~~
	| f'(U_0) + (\vp'(U_0))_{zz} | \leq c_2.
\end{align*} 
 Therefore, we have
\begin{align*}
	\frac{U_{0z}}{\sigma} 
	- \left( f'(U_0) + (\vp'(U_0))_{zz}  \right)
	\geq 
	\dfrac{c_1}{\sigma} - c_2
	\rightarrow
	\infty
	~\text{as}~
	\sigma \downarrow 0,
\end{align*}
implying \eqref{lem_E3_1} on  $J_2$ for $\sigma$ small enough.
\end{proof}

Before we give the rigorous proof that $u^\pm$ are sub- and super-solutions, we first give detailed computations needed in the  sequel. Recall \eqref{star0}.  First, note, with $U_0$ and $U_1$ corresponding to $u^+$, that 
\begin{align}
	\vp(u^+) 
	&= \vp(U_0) 
	+ (\e {U_1} + q) \vp'(U_0)
	+ (\e {U_1} + q)^2 
	\int_0^1 (1 - s) \vp''( U_0 + ( \e {U_1} + q )s ) ds\nonumber\\
	f(u^+)
	&= f(U_0)
	+ (\e {U_1} + q) f'(U_0)
	+ \frac{(\e {U_1} + q)^2 }{2} f''(\theta(x,t)),
	\label{star1}
\end{align}
where $\theta$ is a function satisfying $\theta(x,t) \in \left(U_0, U_0 + \e {U_1} + q(t)\right)$. Straightforward computations yield
\begin{align}
	(u^+)_t
	&= U_{0z} 
	\left(
		\frac{d_t + \e p_t}{\e}
	\right)
	+ \e {U_1}_t
	+ {U_1}_z ( d_t + \e p_t )
	+ q_t
	\nonumber\\ %%%%%%% line 1 %%%%%%%%%% equality 1
	\Delta \vp(u^+)
	&= \nabla \cdot
	\left(
	( \vp(U_0) )_z \frac{\nabla d}{\e}
	+ {U_1}_z \vp'(U_0) \nabla d
	+ \e \nabla {U_1} \vp'(U_0)
	+ (\e {U_1} + q)(\vp'(U_0))_z \frac{\nabla d}{\e}
	+ \nabla R
	\right)
	\nonumber \\ %%%%%%% line 2 %%%%%%%%% equality 2
	&= ( \vp(U_0) )_{zz} \frac{|\nabla d|^2}{\e^2}
	+ (\vp(U_0))_z \frac{\Delta d}{\e}
	\nonumber \\ %%%%%%%% \line 3
	&+ ({U_1}_z \vp'(U_0))_z \frac{| \nabla d|^2}{\e}
	+ {U_1}_z \vp'(U_0) \Delta d
	+ 2 \nabla {U_1}_z \vp'(U_0) \cdot \nabla d
	+ \nabla {U_1} (\vp'(U_0))_z \cdot \nabla d
	+ \e \Delta {U_1} \vp'(U_0)
	\nonumber \\ %%%%%%%%%%%%% line 4
	&+({U_1} \vp'(U_0)_z)_z \frac{|\nabla d|^2}{\e}
	+ q (\vp'(U_0))_{zz} \frac{|\nabla d|^2}{\e^2}
	+ \nabla {U_1} (\vp'(U_0))_z \cdot \nabla d
	\nonumber\\	%%%%%%%%%%%%%% line 5
	&+ (\e {U_1} + q) (\vp'(U_0))_z \frac{ \Delta d }{\e}
	+ \Delta R
	\label{star2}
\end{align}
where $R(x,t) = (\e {U_1} + q)^2 \int_0^1 (1 - s) \vp''( U_0 + ( \e U_1 + q )s ) ds$. Define $r(x,t) = \int_0^1 (1 - s) \vp''( U_0 + ( \e {U_1} + q )s ) ds$. Then, we have 
\begin{align}
	\Delta R(x,t)
	&= \nabla \cdot \nabla 
	\Big{[}
	\Big{(}
	(\e {U_1})^2 + 2 \e q {U_1} + q^2
	\Big{)}r
	\Big{]}
	\nonumber\\	%%%%%%%%%%%%%%%%%% line 1 equality 1
	&= \nabla \cdot 
	\Big{[}
	\Big{(}
	2  \e {U_1} 
	\left(
{U_{1}}_z \nabla d + \e \nabla {U_1} 
	\right)
	+ 2 q
	\left(
	{U_{1}}_z \nabla d + \e \nabla {U_1} 
	\right)
	\Big{]}
	 r(x,t)
	+
	\Big{(}
	(\e{U_1})^2 + 2 \e q {U_1} + q^2
	\Big{)}
	\nabla r(x,t)
	\Big{]}
	\nonumber\\	%%%%%%%%%%%%%%%%%% line 2 equality 2
	&=  
	\left[
	2\left(
{U_1}_z \nabla d + \e \nabla {U_1} 
	\right)^2
	+  2 \e {U_1} 
	\left(
	U_{1zz} \frac{| \nabla d |^2 }{\e} 
	+ {U_1}_z \Delta d
	+ 2\nabla {U_1}_z \cdot \nabla d
	+ \e \Delta {U_1} 
	\right)
	\right] r(x,t)
	\nonumber\\	%%%%%%%%%%%%% line 3
	& + 2q \left(
	U_{1zz} \frac{| \nabla d |^2 }{\e} 
	+ {U_1}_z \Delta d
	+ 2\nabla {U_1}_z \cdot \nabla d
	+ \e \Delta {U_1} 
	\right) 
	r(x,t)
	\nonumber\\	%%%%%%%%%%% line 4
	& + 2 
	\Big{[}
	2  \e {U_1} 
	\left(
	{U_{1}}_z \nabla d + \e \nabla {U_1} 
	\right)
	+ 2 q
	\left(
	{U_{1}}_z \nabla d + \e \nabla {U_1} 
	\right)
	\Big{]}
	\nabla r(x,t)
	\nonumber\\ %%%%%%%%%% line 5
	& +
	\Big{(}
	(\e {U_1})^2 + 2 \e q {U_1} + q^2
	\Big{)}
	\Delta r(x,t)
	\label{star3}
\end{align}
where
\begin{align*}
	\nabla r(x,t) 
	&= \int_0^1 (1 - s) \vp'''( U_0 + ( \e {U_1} + q) s )
	\left(
	\left(
	U_{0} + \e  U_{1} s
	\right)_z
	\frac{\nabla d}{\e}
	+ \e \nabla {U_1} s	
	\right)
	ds \\
	\Delta r(x,t)
	&= \int_0^1 (1 - s) \vp'''( U_0 + ( \e {U_1} + q) s )
	\left(
	(U_0 + \e {U_1} s)_{z} \frac{ \Delta d }{\e}
	\right.
	\\
	&\left.
	+
	(U_0 + \e {U_1} s)_{zz} \frac{ | \nabla d |^2 }{\e^2}
	+
	( 2 \nabla {U_1}_z \cdot \nabla d
	+ \e \Delta {U_1})s
	\right)
	ds
	\\
	&+ \int_0^1 (1 - s) \vp^{(4)}( U_0 + ( \e {U_1} + q) s )
	\left(
	(U_{0} + \e  {U_{1}} s)_z
	\frac{\nabla d}{\e}
	+ \e \nabla {U_1} s
	\right)^2
	ds.
\end{align*}
Define $l(x,t), r_i(x,t)$ for $i = 1,2,3$ as follows:
\begin{align*}
	l(x,t) 
	&= 
	U_{1zz} \frac{| \nabla d |^2 }{\e} 
	+ {U_1}_z \Delta d
	+ 2\nabla {U_1}_z \cdot \nabla d
	+ \e \Delta {U_1} 
	\\
	r_1(x,t)
	&= \left[
	2\left(
	{U_1}_z \nabla d + \e \nabla {U_1} 
	\right)^2
	+  2 \e {U_1} 
	l(x,t)
	\right] r(x,t)
	+ 4 \e {U_1} 
	\left(
	{U_{1}}_z \nabla d + \e \nabla {U_1} 
	\right) \nabla r(x,t)
	+ ( \e {U_1} )^2 \Delta r(x,t)
	\\
	r_2(x,t)
	&= 2 q l(x,t) r(x,t)	
	+4 q
	\left(
{U_{1}}_z \nabla d + \e \nabla {U_1} 
	\right)
	\nabla r(x,t)
	+ 2 \e q {U_1} \Delta r(x,t)
	\\
	r_3(x,t)
	&= q^2 \Delta r(x,t).
\end{align*}
Thus, 
\begin{align}
\label{star4}
\Delta R=r_1+r_2+r_3.
\end{align}

We have the following properties for $r_i$.
\begin{lem}\label{Lem_remiander_bound}
 There exists $C_r > 0$ independent of $\e$ such that 
\begin{eqnarray} \label{ineq_r}
	|r_1| \leq C_r, ~~~
	|r_2| \leq \frac{q}{\e} C_r, ~~~
	|r_3| \leq \frac{q^2}{\e^2} C_r.
\end{eqnarray}
\end{lem}

\begin{proof}
 Note that, by Lemmas \ref{Lem_U0_bound},  \ref{Lem_U1_bound} and (\ref{eqn_cond_elc}) the term $U_a := U_0 + (\e {U_1} + q)s$ is uniformly bounded.  Hence, the terms $\vp''(U_a), \vp'''(U_a), \vp^{(4)}(U_a)$ are uniformly bounded, and in particular $r$ is bounded. By similar reasoning for $\nabla r$ and $\Delta r$, it follows that there exists some positive constants $c_\nabla, c_\Delta$ such that 
$$
	| \nabla r | \leq \frac{c_\nabla}{\e}, ~~~
	| \Delta r | \leq \frac{c_\Delta}{\e^2}.
$$
 
\noindent   Moreover, by Lemmas \ref{Lem_U0_bound}, \ref{Lem_U1_bound} there exists a positive constant $c_l$ such that 
$$
	|l(x,t)| 
	\leq \frac{c_l}{\e}.
$$
Combining these estimates yields \eqref{ineq_r}.

\end{proof}

\noindent Let $\sigma$ a fixed constant satisfying 
\begin{align}\label{cond_sigma}
	0 < \sigma \leq \min \{ \sigma_0,\sigma_1,\sigma_2 \},
\end{align}
where $\sigma_0$ is the constant defined in Lemma \ref{Lem_E3_bound}, and $\sigma_1$ and $\sigma_2$ are given by
\begin{align}\label{cond_sigma2}
	\sigma_1 = \frac{1}{2(\beta + 1)}, ~~~
	\sigma_2 = \frac{\beta}{( F + C_r) (\beta + 1)}, ~~~ F = ||f''||_{L^\infty(\alpha_- -1, \alpha_+ + 1)}.
\end{align}
Note that, since $\sigma < \sigma_1$ and \eqref{eqn_cond_elc}, we have
\begin{align*}
	\alpha_- - 1 
	\leq
	|u^\pm|
	\leq
	\alpha_+ + 1.
\end{align*}

\begin{lem}\label{Lem_Prop_subsuper}
	Let  $\beta$ be given by (\ref{cond_beta}) and let $\sigma$ satisfy (\ref{cond_sigma}). Then, there exists $\e_0 > 0$ and a positive constant $C_p$, which does not depend on $\e$, such that 
\begin{align}\label{eqn_Prop_subsuper}
	\begin{cases}
		\mathcal{L} (u^-) 
		< 
		- C_p
		< 
		C_p
		< 
		\mathcal{L}(u^+)
		&
		\text{ in } \overline{D} \times [0,T]
		\\
		\displaystyle{\frac{\partial u^-}{\partial \nu}
		= \frac{\partial u^+}{\partial \nu}}
		= 0
		&
		\text{ on } \partial D \times [0,T]
	\end{cases}
\end{align}
for every $\e \in (0, \e_0)$.
\end{lem}

\begin{proof}
In the following, we only show that $u^+$ is a super solution; one can show that $u^-$ is a sub-solution in a similar way.

Combining the computations above in \eqref{star1}, \eqref{star2}, \eqref{star3} and \eqref{star4}, we obtain
\begin{align*}
	\mathcal{L}u^+
	&= (u^+)_t - \Delta (\vp(u^+)) - \frac{1}{\e^2} f(u^+)
	\\	
	&= { E_1+E_2+E_3+E_4+E_5+E_6,}	
\end{align*}
{
where }
\begin{align*}
	E_1
	&= 
	- \frac{1}{\e^2} 
	\left(
		(\vp(U_0))_{zz} | \nabla d |^2 + f(U_0)
	\right)
	- \frac{| \nabla d |^2 - 1}{\e^2} q(\vp'(U_0))_{zz}	 
	- \frac{| \nabla d |^2 - 1}{\e} ({U_1} \vp'(U_0))_{zz}
	\\
	E_2
	&= 
	\frac{1}{\e} U_{0z} d_t
	- \frac{1}{\e}
	\left(
		(\vp(U_0))_z \Delta d
		+ ({U_1}_z \vp'(U_0))_z 
		+ ({U_1}\vp'(U_0)_z)_z
		+{U_1} f'(U_0)
	\right)
	\\
	E_3
	&= 
	[ U_{0z} p_t + q_t ] 
	- \frac{1}{\e^2}
	\left[
		q f'(U_0) 
		+ q (\vp'(U_0))_{zz} 
		+ \frac{q^2}{2} f''(\theta)
	\right]
	- r_3(x,t)
	\\
	E_4
	&= 
	\e {U_1}_z p_t 
	- \frac{q}{\e} 
	\Big{[} 
		(\vp'(U_0))_z \Delta d + {U_1}  f''(\theta) 
	\Big{]}
	- r_2(x,t)
	\\
	E_5
	&= 
	\e {U_1}_t
	- \e \Delta {U_1} \vp'(U_0)
	\\
	E_6
	&= 
	{U_1}_z d_t 
	- 2 \nabla {U_1}_z \vp'(U_0) \cdot \nabla d
	- 2 \nabla {U_1} (\vp'(U_0))_z \cdot \nabla d
	- ( {U_1} \vp'(U_0))_z \Delta d
	- r_1(x,t)
	- \frac{({U_1})^2}{2} f''(\theta).
\end{align*}

\vskip .1cm
{\it Estimate of the term $E_1$.}
Using (\ref{eqn_AsymptExp_U0}) we write $E_1$ in the form
$$
	E_1 
	= -\frac{| \nabla d |^2 - 1}{\e^2}
	\big(
	(\vp(U_0))_{zz} + q(\vp'(U_0))_{zz}
	\big)
	-\frac{| \nabla d |^2 - 1}{\e}
	({U_1}\vp'(U_0))_{zz}.
$$
We only consider the term 
$
	e_1 
	:= \displaystyle{\frac{| \nabla d |^2 - 1}{\e}}
	({U_1}\vp'(U_0))_{zz}
$
; the other terms can be bounded similarly. In the region where $|d| \leq d_0$, we have $| \nabla d | = 1$ so that $e_1 = 0$. If, however $|\nabla d| \neq 1$, we have
$$
	\frac{|({U_1}\vp'(U_0))_{zz}|}{\e}
	\leq \frac{\hat{C}_1}{\e} e^{- \lambda_1
	\left| \frac{d}{\e} + p(t) 
	\right|
	}
	\leq \frac{\hat{C}_1}{\e} e^{- \lambda_1
	\left[ \frac{d_0}{\e} - p(t)
	\right]}
	\leq \frac{\hat{C}_1}{\e} e^{- \lambda_1 
	\left[ \frac{d_0}{\e} - (1 + e^{LT} + K)
	\right]}.
$$
Choosing $\e_0$ small enough such that 
$$
	\frac{d_0}{2\e_0} 
	- 
	\Big(
		1 + e^{LT} + K
	\Big)
	\geq 0,
$$
we deduce 
$$
	\frac{|({U_1} \vp'(U_0))_{zz}|}{\e}
	\leq \frac{\hat{C}_1}{\e} e^{- \lambda_1 \frac{d_0}{2 \e}} 
	\rightarrow 0 
	\text{ as }
	\e \downarrow 0.
$$
Thus, $\tfrac{1}{\e} |({U_1} \vp'(U_0))_{zz}|$ is uniformly bounded, so that there exists $\hat{C}_2$ independent of $\e, L $ such that 
$$
	| e_1 | \leq \hat{C}_2.
$$
Finally, as a consequence, we deduce that there exists $\tilde{C}_1$ independent of $\e, L $ such that 
\begin{align}\label{eqn_E1_bound}
	| E_1 | \leq \tilde{C}_1.
\end{align}

\vskip .1cm
{\it Estimate of the term $E_2$.}
Using (\ref{eqn_U1_bar}), we write $E_2$ in the form
$$
	E_2 
	= \frac{1}{\e} U_{0z} d_t
	- \frac{1}{\e} \lambda_0 U_{0z} \Delta d
	= \frac{U_{0z}}{\e} (d_t - \lambda_0 \Delta d).
$$
Applying Lemma \ref{Lem_d_bound}, \ref{Lem_U0_bound} and  \ref{Lem_U1_bound} gives 
$$
	|E_2| 
	\leq C_d \hat{C}_0 \frac{|d|}{\e} e^{- \lambda _1
	\left|
	\frac{d}{\e} + p
	\right|
	} 
	\leq C_d \hat{C}_0
	\max_{\xi \in \mathbb{R}} | \xi | e^{-\lambda_1 |\xi + p|}.
$$
Note that $\max_{\xi \in \mathbb{R}} | \xi | e^{-\lambda_1 |\xi + p|} \leq |p| + \frac{1}{\lambda_1}$ (cf.\ \cite{Danielle2018}). Thus, there exists $\tilde{C}_2$ such that 
\begin{align}\label{eqn_E2_bound}
	|E_2| \leq \tilde{C}_2(1 + e^{LT}).
\end{align}

\vskip .1cm
{\it Estimate of the term $E_3$.}
Substituting $p_t = \dfrac{q}{\e^2 \sigma}$ and then replacing $q$ by its explicit form (cf.\ \eqref{star0}) gives
\begin{align*}
	E_3
	&= \frac{q}{\e^2\sigma}
	\left[
		U_{0z} - \sigma ( f'(U_0) + (\vp'(U_0) )_{zz} ) - \sigma q
		\left(
		\frac{1}{2}f''(\theta) + \frac{\e^2}{q^2} r_3
		\right)
	\right]
	+ q_t
	\\	%%%%%%%% line 1
	&= \frac{1}{\e^2}
	\left(
		\beta e^{- \frac{\beta t }{\e^2}} + \e^2 L e^{Lt}
	\right)
	\left[
		U_{0z} - \sigma ( f'(U_0) + (\vp'(U_0) )_{zz} ) - \sigma^2 (\beta e^{- \frac{\beta t}{\e^2}} + L \e^2 e^{Lt})
		\left(
		\frac{1}{2}f''(\theta) + \frac{\e^2}{q^2} r_3
		\right)
	\right]
	\\	%%%%%%%% line 2
	& - \frac{1}{\e^2} \sigma \beta^2 e^{ - \frac{\beta t}{\e^2}} 
	+ \e^2 \sigma L^2 e^{Lt}
	\\	%%%%%%%% line 3
	&= \frac{1}{\e^2} \beta e^{- \frac{\beta t}{\e^2}}(I - \sigma\beta)
	+ L e^{Lt} [I + \e^2 \sigma L]
\end{align*}
where
$$
	I 
	:= U_{0z} - \sigma ( f'(U_0) + (\vp'(U_0) )_{zz} ) - \sigma^2 (\beta e^{- \frac{\beta t}{\e^2}} + L \e^2 e^{Lt})
	\left(
	\frac{1}{2}f''(\theta) + \frac{\e^2}{q^2} r_3 
	\right).
$$
Applying Lemma \ref{Lem_E3_bound}, using \eqref{eqn_cond_elc} and \eqref{cond_sigma}, yields
\begin{eqnarray*}
	I &\geq & 3 \sigma \beta 
	- \sigma \sigma_2  \left(\beta  + L \e^2 e^{Lt} \right)
	\left(
	|f''(\theta)| + \frac{\e^2}{q^2} |r_3|
	\right)\\
	 &\geq &  3 \sigma \beta 
	- \sigma \sigma_2  \left(\beta +1 \right)
	\left(
	|f''(\theta)| + \frac{\e^2}{q^2} |r_3|
	\right)\\
	&\geq& 2 \sigma \beta,
\end{eqnarray*}
where the last inequality follows from \eqref{cond_sigma2}. This implies that 
\begin{align}\label{eqn_E3_bound}
	E_3 
	\geq \frac{\sigma \beta^2}{\e^2} e^{- \frac{\beta t }{\e^2}}
	+ 2 \sigma \beta L e^{Lt}.
\end{align}

\vskip .1cm
{\it Estimate of the term $E_4$.}
Substituting again $p_t = \dfrac{q}{\e^2 \sigma}$, with $q$ in its explicit form \eqref{star0} gives
\begin{align*}
	E_4 
	&= \frac{q}{\e \sigma}
	\left( {U_1}_z
	- \sigma((\vp'(U_0))_z \Delta d +U_1 f''(\theta))
	- \sigma\frac{\e}{q} r_2
	\right)\\
	&= \frac{1}{\e}
	\left(
		\beta e^{-\frac{ \beta t }{\e^2}} + \e^2 L e^{Lt}
	\right)
	\left({U_1}_z
		- \sigma((\vp'(U_0))_z \Delta d + {U_1} f''(\theta))
		- \sigma\frac{\e}{q} r_2
	\right).
\end{align*}
Applying Lemma \ref{Lem_d_bound}, \ref{Lem_U0_bound}, \ref{Lem_U1_bound} and \ref{Lem_remiander_bound} gives the uniform boundedness of the last factor in parenthesis. Thus, there exists a constant $\tilde{C}_4$ such that
\begin{align}\label{eqn_E4_bound}
	|E_4| 
	\leq \tilde{C}_4 
	\frac{1}{\e}
	\left(
		\beta e^{-\frac{\beta t }{\e^2}} + \e^2 L e^{Lt}
	\right).
\end{align}

\vskip .1cm
{\it Estimate of the terms $E_5$ and $E_6$.}
Applying Lemma \ref{Lem_d_bound}, \ref{Lem_U0_bound} and \ref{Lem_U1_bound}, it follows that there exists $\tilde{C}_5$ such that 
\begin{align}\label{eqn_E5_bound}
	|E_5| + |E_6|
	\leq \tilde{C}_5.
\end{align}

\vskip .1cm
{\it Combination of the above estimates.}
Collecting the estimates (\ref{eqn_E1_bound}),(\ref{eqn_E2_bound}),(\ref{eqn_E3_bound}),(\ref{eqn_E4_bound}),(\ref{eqn_E5_bound}), we obtain 
\begin{align*}
	\mathcal{L}(u^+)
	&\geq
	\left[
		\frac{\sigma \beta^2}{\e^2}
		- \tilde{C}_4 \frac{\beta}{\e}
	\right]e^{- \frac{\beta t }{\e^2}}
	+
	\left[
		2 \sigma \beta L
		- \e \tilde{C}_4 L
		-\tilde{C}_2
	\right]e^{Lt}
	- \tilde{C}_1 - \tilde{C}_2 - \tilde{C}_5
	\\
	& \geq 
	\left[
		\frac{\sigma \beta^2}{\e^2}
		- \tilde{C}_4 \frac{\beta}{\e}
	\right]e^{- \frac{\beta t }{\e^2}}
	+
	\left[
		\frac{2 \sigma \beta L}{3}
		- \e \tilde{C}_4 L
	\right]e^{Lt}
	\\
	&+
	\left[
		\frac{2 \sigma \beta L}{3}
		-\tilde{C}_2
	\right]e^{Lt}
	+
	\left[
		\frac{2 \sigma \beta L}{3}
		-\tilde{C}_6
	\right]
\end{align*}
where $\tilde{C}_6 = \tilde{C}_1 + \tilde{C}_2 + \tilde{C}_5$. Choose $\e_0$ small enough and $L$ large enough so that
\begin{align*}
	\sigma \beta 
	> 
	3 \tilde{C}_4 \e_0
	,~
	\sigma \beta L
	> 3 \max \{ \tilde{C}_2, \tilde{C}_6 \}.
\end{align*}
Then, we deduce that there exists a positive constant $C_p$, independent of $\e$, such that $\mathcal{L}(u^+) \geq C_p$. 
\end{proof}

\subsection{Proof of Theorem \ref{Thm_Propagation}}
\label{proof_thm_prop}

The proof of Theorem \ref{Thm_Propagation} is divided in two steps: (i) For large enough $K>0$, we prove that $u^-(x,t) \leq u^\e(x,t + t^\e) \leq u^+(x,t)$ for $x \in \overline{D}, t \in [0,T - t^\e]$ and (ii) we employ (i) to show the desired result.

\vskip .1cm
{\it Step 1.}  Fix $\sigma, \beta$ as in (\ref{cond_beta}), (\ref{cond_sigma}). 
 Without loss of generality, we may assume that 
\begin{align*}
	0 < \eta < \min \left\{ \eta_0, \sigma\beta  \right\}.
\end{align*} 
Theorem \ref{Thm_Generation} implies the existence of constants $\e_0$ and $M_0$ such that (\ref{Thm_generation_i})-(\ref{Thm_generation_iii}) are satisfied. 
 Conditions (\ref{cond_gamma0_normal}) and (\ref{cond_u0_inout}) imply that there exists a positive constant $M_1$ such that 
\begin{align*}
	&\text{if } d(x,0) \leq - M_1 \e, ~~ \text{ then } u_0(x) \leq \alpha - M_0 \e,
	\\
	&\text{if } d(x,0) \geq  M_1 \e, ~~ \text{ then } u_0(x) \geq \alpha + M_0 \e.
\end{align*}
Hence, we deduce, by applying (\ref{Thm_generation_i}), (\ref{Thm_generation_iii}), that
\begin{align*}
	u^\e(x,t^\e) 
	\leq H^+(x)
	:=
	\begin{cases}
		\alpha_+ + \frac{\eta}{4}
		&
		~~
		\text{ if }
		d(x,0) \geq - M_1 \e
		\\
		\alpha_- + \frac{\eta}{4}
		&
		~~
		\text{ if }
		d(x,0) < - M_1 \e.
	\end{cases}
\end{align*}
Also, by applying (\ref{Thm_generation_i}), (\ref{Thm_generation_ii}),
\begin{align*}
	u^\e(x,t^\e) 
	\geq H^-(x)
	:=
	\begin{cases}
		\alpha_+ - \frac{\eta}{4}
		&
		~~
		\text{ if }
		d(x,0) > M_1 \e
		\\
		\alpha_- - \frac{\eta}{4}
		&
		~~
		\text{ if }
		d(x,0) \leq M_1 \e.
	\end{cases}
\end{align*} 
 
Next,  we fix a sufficient large constant $K$ such that 
\begin{align*}
	U_0(M_1 - K) \leq \alpha_- + \frac{\eta}{4}
	~~~\text{and}
	~~~
	U_0(- M_1 + K) \geq \alpha_+ - \frac{\eta}{4}.
\end{align*}
For such a constant $K$, Lemma \ref{Lem_Prop_subsuper} implies the existence of coefficients $\e_0$ and $L$ such that the inequalities in (\ref{eqn_Prop_subsuper}) holds. 
We claim that
\begin{align}\label{thm_prop_proof1}
	u^-(x,0) \leq H^-(x)
	\leq
	H^+(x)
	\leq
	u^+(x,0).
\end{align}
We only prove the last inequality since the first inequality can be proved similarly. By Lemma \ref{Lem_U1_bound}, we have 
$| {U_1} | \leq \hat{C}_1$. Thus, we can choose $\e_0$ small enough so that, for $\e\in (0,\e_0)$, we have $ \e \hat{C}_1 \leq \dfrac{\sigma \beta}{4}$.  Then, noting \eqref{star0},
\begin{align*}
	u^+(x,0) 
	&\geq 
	U_0
	\left(
		\frac{d(x,0) + \e p(0)}{\e}
	\right)
	-
	\e \hat{C}_1 + \sigma \beta + \e^2 \sigma L 
	\\
	&>
	U_0
	\left(
		\frac{d(x,0)}{\e} + K
	\right)
	+ \frac{3}{4} \eta.
\end{align*}
In the set $\{ x \in D : d(x,0) \geq - M_1\e \}$, the inequalities above, and the fact that $U_0$ is an increasing function imply
\begin{align*}
	u^+(x,0) 
	>
	U_0(- M_1 + K) + \frac{3}{4} \eta
	\geq 
	\alpha_+ + \frac{\eta}{2}
	> H^+(x).
\end{align*}
Moreover, since $U_0 \geq \alpha_-$ in the set $\{ x \in D : d(x,0) < - M_1\e \}$, we have
\begin{align*}
	u^+(x,0)
	> \alpha_- + \frac{3}{4} \eta
	> H^+(x).
\end{align*}
Thus, we proved the first inequality in \eqref{thm_prop_proof1} above.

The inequalities \eqref{thm_prop_proof1} and Lemma \ref{Lem_Prop_subsuper} {now} permit to apply the comparison principle Lemma \ref{lem_comparison}, so that we have 
\begin{align}\label{compprinciple}
	u^-(x,t) \leq u^\e(x,t + t^\e) \leq u^+(x,t)
	~~ \text{ for } ~~	
	x \in \overline{D}, t \in [0,T - t^\e].
\end{align}

\vskip .1cm
{\it Step 2.} Choose $C > 0$ so that 
\begin{align*}
	U_0(C - e^{LT} - K)  \geq \alpha_+ -  \frac{\eta}{2}
	~~ \text{and} ~~
	U_0( - C + e^{LT} + K) \leq \alpha_- + \frac{\eta}{2}.
\end{align*}
Then, we deduce from \eqref{compprinciple}, noting \eqref{star0}, that 
\begin{align*}
	&\text{if}~
	d(x,t) \geq  \e C,
	~\text{then }
	u^\e(x,t + t^\e)
	\geq \alpha_+ - \eta 
	\\
	&\text{if}~
	d(x,t) \leq - \e C,
	~\text{then }
	u^\e(x,t + t^\e)
	\leq \alpha_- + \eta
\end{align*}
and since $\alpha_\pm \pm \eta$ are respectively sub- and super-solutions of $(P^\e)$, we conclude that 
\begin{align*}
	u^\e(x,t + t^\e) \in [\alpha_- - \eta, \alpha_+ + \eta]
\end{align*}
for all $(x,t) \in D \times [0, T - t^\e], \e \in (0,\e_0)$. \hfill \qed

\begin{rmk}\label{rmk_thm13}
	These sub and super solutions guarantee that $u^\e \simeq \alpha_+$(respectively, $u^\e \simeq \alpha_-$) for $d(x,t) \geq c$(respectively, $d(x,t) \leq -c$) with $t > \rho t^\e, \rho > 1$ and $\e > 0$ small enough. In fact, by the definition of $q(t)$, we expect
\begin{align*}
	\e U_1 \pm q(t)
	= \mathcal{O}(\e)
\end{align*}
for $t > (\rho - 1) t^\e$. Also, by Lemma \ref{Lem_U0_bound}, we expect
\begin{align*}
	0 < U_0(z) - \alpha_- < \tilde{c} \e
	~\text{for}~ 
	z > \dfrac{c}{\e}
	,~~
	0 <  \alpha_+ - U_0(z) < \tilde{c} \e
	~\text{for}~ 
	z < - \dfrac{c}{\e}.
\end{align*}
 These estimates yield that there exists a positive constant $c'$ such that 
\begin{align*}
	|u^\e(x,t) - \alpha_+| \leq c'\e
	~\text{for}~
	d(x,t) > c
	,~~
	|u^\e(x,t) - \alpha_-| \leq c'\e
	~\text{for}~
	d(x,t) < - c
\end{align*}
for $t > \rho t^\e.$

\end{rmk}

\section{Proof of Theorem \ref{thm_asymvali}}\label{section_5}

We now introduce the concept of an eternal solution. A solution of an evolution equation is called \textit{eternal} if it is defined for all positive and negative times. In our problem, we study the nonlinear diffusion problem 
\begin{align}\label{eqn_entire}
	w_\tau 
	= \Delta \vp(w) + f(w),
	~~ ((z', \zn), \tau) \in \R^N \times \R,
\end{align}
{
where $z'\in \R^{N-1}$ and $z^{(N)}\in \R$.}
In order to prove Theorem \ref{thm_asymvali}, we first present two lemmas.

\begin{lem}\label{lem_locregularity}
	Let $S$ be a domain of $\mathbb{R}^N \times \R$ and let
	$u$ be a bounded function on $S$ satisfying 
\begin{align}  \label{eq:NAC}
	u_t = \Delta \vp(u) + f(u), ~~ (x,t) \in S,  
\end{align}
where $\vp, f$ satisfy conditions \eqref{cond_f_bistable}, \eqref{cond_phi'_bounded}. 
Then, for any smooth bounded subset $S' \subset S$ separated from $\partial S$ by a positive constant $\tilde{d}$ we have
\begin{align}  \label{eq:C21}
	\Vert u \Vert_{C^{2 + \theta, 1 + \theta/2}(\overline{S'})}
	\leq C',
\end{align}
for any positive constants $0 < \theta < 1$ and $C'= C'(\|u\|_{L^\infty(S)})$ which depends on 
$\|u\|_{L^\infty(S)}$, $\fa, f$,  $\tilde{d}, \theta$ and the size of $S'$, where
\begin{align*}
	\Vert u \Vert_{C^{k + \theta, k' + \theta'}(\overline{S'})}
	&=
	\Vert u \Vert_{C^{k,k'}(\overline{S'})} 
	+ \sum_{i,j = 1}^N 
	\sup_{(x,t),(y,t) \in S', x \neq y}
	\left\{
		\dfrac{|D^k_x u(x,t) - D^k_x u(y,t)|}{|x - y|^\theta}
	\right\}
	\\
	&
	+ 
	\sup_{(x,t),(x,t') \in S', t \neq t'}
	\left\{
		\dfrac{|D^{k'}_tu(x,t) - D^{k'}_tu(x,t')|}{|t - t'|^{\theta'}}
	\right\}
\end{align*}
where $k, k'$ are non-negative integers and $0 < \theta, \theta' < 1$.
\end{lem}

\begin{proof}

Since $S'$ is separated from $\partial S$ by a positive distance, we can find  subsets $S_1, S_2 $ such that $S' \subset S_2 \subset S_1 \subset S$ and such that $\partial S, \partial S', \partial S_i$ are separated by a positive distance less than $\tilde{d}$.
By condition \eqref{cond_phi'_bounded} the regularity of $u(x,t)$ is the same as the regularity of $v(x,t) = \vp(u(x,t))$. Note that by \eqref{eq:NAC} $v$ satisfies
\begin{align*}
	v_t
	=
	\vp'(\vp^{-1}(v))  [ \Delta v +g(v) ]
	~,
	g(s) = f(\vp^{-1}(s))
\end{align*}
on $S$. 
By Theorem 3.1 p.\ 437-438 of \cite{LOVA1988}, there exists a positive constant $c_1$ such that
\begin{align*}
	| \nabla v |
	\leq c_1
	\text{ in }
	S_1
\end{align*}
where $c_1$ depends only on $N, \vp, ||u||_{L^\infty(S)}$ and the distance between $S$ and $S_1$. This, together with Theorem 5, p 122 of \cite{Krylov2008}, imply that 
\begin{align*}
	\Vert v \Vert_{W^{2,1}_p(S_2)}
	\leq 
	c_2(\Vert v \Vert_{L^p(S_1)} + \Vert \vp'(\vp^{-1}(v)) g(v) \Vert_{L^p(S_1)})
\end{align*}
for any $p > {N + 2}$ where $c_2$ is a constant that depends on $c_1, p, N, \vp$. With this, by fixing $p$ large enough, the Sobolev embedding theorem in chapter 2, section 3 of \cite{LOVA1988} yields 
\begin{align*}
	\Vert v \Vert_{C^{1 + \theta, (1 + \theta)/2}(S_2)}
	\leq c_3 \Vert v \Vert_{W^{2,1}_p(S_2)}
\end{align*}
where $0 < \theta < 1 - \frac{N + 2}{p}$ and $c_3$ depends on $c_2$ and $p$. This implies that $\vp'(\vp^{-1}(v)), g(v)$ are bounded uniformly in $C^{1 + \theta, (1 + \theta)/2}(S_2)$. Therefore, by Theorem 10.1 p 351-352 of \cite{LOVA1988} we obtain 
\begin{align*}
	\Vert v \Vert_{C^{2 + \theta, 1 + \theta/2}(S')}
	\leq c_4 \Vert v \Vert_{C^{1 + \theta, (1 + \theta)/2}(S_2)}
\end{align*}
where $c_4$ depends on $c_2, f$ and $\vp$.
\end{proof}

\begin{rmk}\label{rmk_locreg}
 Lemma \ref{lem_locregularity} implies uniform $C^{2,1}$ boundedness of the entire solution $w$ in the whole space.  This can be derived as follows: Let
\begin{align*}
	S_{(a,b)} = \{(x,t) \in \R^N \times \R, |x - a|^2 + (t - b)^2 \leq 2 \},~
	S'_{(a,b)} = \{(x,t) \in \R^N \times \R, |x - a|^2 + (t - b)^2 \leq 1 \}
\end{align*} 
where ${(a,b)} \in \R^N \times \R$. Then, Lemma \ref{lem_locregularity} implies uniform $C^{2,1}$ boundedness of $w$ within $S'_{(a,b)}$ where the upper bound is fixed by \eqref{eq:C21}. Since this upper bound is independent to the choice of $(a,b)$, we have uniform $C^{2,1}$ bound of $w$ in the whole space.
\end{rmk}

Next, we present a result inspired by a similar one in \cite{BH2007}.

\begin{lem}\label{lem_entire}

 Let $w((z',\zn),\tau)$ be a bounded eternal solution of \eqref{eqn_entire} satisfying 
\begin{align}\label{lem_entire_1}
	\liminf_{\zn \rightarrow - \infty} \inf_{z' \in \R^{N - 1}, \tau \in \R} w((z',\zn),\tau) = \alpha_-
	,~~
	\limsup_{\zn \rightarrow \infty} \sup_{z' \in \R^{N - 1}, \tau \in \R} w((z',\zn),\tau) = \alpha_+,
\end{align}
where $z' = (z^{(1)}, z^{(2)}, \cdots z^{(N - 1)})$. Then, there exists a constant $z^* \in \R$  such that 
\begin{align*}
	w((z',\zn),\tau) = U_0(\zn - z^*).
\end{align*}

\end{lem}

\begin{proof}
We prove the lemma in two steps. First we show $w$ is an increasing function with respect to the $\zn$ variable. Then, we prove that $w$ only depends on $\zn$, which means that  there exists a function $\psi : \R \rightarrow (\alpha_-, \alpha_+)$ such that 
\begin{align*}
	w((z',\zn),\tau) = \psi(\zn), \ \  ((z',\zn), \tau) \in \R^N \times \R.
\end{align*}
From the increasing property with respect to $\zn$, this allows us to identify $\psi$ as the unique standing wave solution $U_0$ of the problem \eqref{eqn_AsymptExp_U0} up to a translation factor $z^*$.

 We deduce from \eqref{lem_entire_1} that there exist $A > 0$ and $\eta \in (0, \eta_0)$ such that 
\begin{align}\label{lem_entire_2}
	\begin{cases}
	\alpha_+ - \eta
	\leq
	w((z',\zn),\tau) 
	\leq  
	\alpha_+ + \eta
	,
	&~~ \zn \geq A
	\\
	\alpha_- - \eta
	\leq
	w((z',\zn),\tau) \leq \ \alpha_- + \eta
	,
	&~~ \zn \leq - A
	\end{cases}
\end{align}
where $\eta_0$ is defined in \eqref{cond_mu_eta0}.

Let $\tilde{\tau} \in \R, \rho \in \R^{N - 1}$ be arbitrary. Define 
\begin{align*}
	w^s((z',\zn),\tau) := w((z' + \rho, \zn + s), \tau + \tilde{\tau})
\end{align*}
where $s \in \R$. Fix $\chi \geq 2 A$ and define
\begin{align}\label{lem_entire_9}
	b^* 
	:= 
	\inf
	\left\{
		b > 0:  \vp(w^\chi) + b \geq \vp(w) ~ \text{in}~ \R^N \times \R
	\right\}.
\end{align}
We will prove that $b^* = 0$, which will imply that $w^\chi \geq w$ in $\R^N \times \R$ since $\vp$ is a strictly increasing function. To see this, we assume that this does not hold, that is $b^* > 0$. 
Note, by \eqref{lem_entire_1} and \eqref{lem_entire_2} we have
\begin{align}\label{lem_entire_3}
	w^\chi \geq \alpha_+ - \eta > \alpha_- + \eta \geq w
	~\text{if}~
	\zn = -A,~
	\lim_{\zn \rightarrow \pm \infty}  
	\vp(w^\chi) - \vp(w) \to 0.
\end{align}

Let $E = \{ (x,t) \in \R^N \times \R, \vp(\wi) - \vp(\wic) > 0\}$. Define a function $Z$ on $E$ as follows
\begin{align*}
	Z((z', \zn), \tau) 
	&:= e^{-C_Z \tau}[\vp(\wi) - \vp(\wic)]((z',\zn),\tau)
	,
	\\
	C_Z
	&:=
	\max
	\left(
	\sup_{x \in {E}}
	\dfrac{[\vp'(\wi) - \vp'(\wic)] \Delta \vp(\wi) 
	+ [\vp'(\wi)f(\wi) - \vp'(\wic)f(\wic)]}{\vp(\wi) - \vp(\wic)}
	,
	0
	\right) \geq 0.
\end{align*}
Note that $C_Z$ is bounded, since $\wic$ is bounded uniformly in $C^{2,1}(E)$ by Remark \ref{rmk_locreg}
 and in view of \eqref{cond_f_bistable} and \eqref{cond_phi'_bounded} we have
\begin{align}\label{lem_entire_10}
	\lim_{x \to y}
	\dfrac{\vp'(x) - \vp'(y)}{\vp(x) - \vp(y)}
	= \dfrac{\vp''(y)}{\vp'(y)}
	< \infty,
	\lim_{x \to y}
	\dfrac{\vp'(x)f(x) - \vp'(y)f(y)}{\vp(x) - \vp(y)}
	=
	\dfrac{(\vp'f)'(y)}{\vp'(y)}
	< \infty.
\end{align}
Direct computations give
\begin{align*}
	Z_\tau - \vp'(\wic) \Delta Z
	&=
	e^{-C_Z\tau}\vp'(\wi)[\Delta \vp(\wi) + f(\wi)]
	-
	e^{-C_Z\tau}\vp'(\wic)[\Delta \vp(\wic) + f(\wic)]
	\\
	&- C_Z Z
	- e^{-C_Z\tau}\vp'(\wic) [\Delta \vp(\wi) - \Delta \vp(\wic)]
	\\
	&=
	\Big(
	[\vp'(\wi) - \vp'(\wic)]\Delta \vp(\wi)
	+[\vp'(\wi)f(\wi) - \vp'(\wic)f(\wic)]
	\Big)
	e^{-C_z\tau}
	- C_Z Z
	\\
	&\leq
	C_Z Z - C_Z Z = 0
\end{align*}
in $E$.
Then, the maximum principle \cite{PW1984} Theorem 5 p.173 yields that the maximum of $Z$ is located at the boundary of $E$. By the definition of $E$, $Z = 0$ on the boundary of $E$ which implies $Z \leq 0$ in $E$. This contradicts the definition of $E$. Thus, we conclude that $b^* = 0$.

Next, we prove that $w \leq w^\chi$ for any $\chi > 0$ (see \eqref{lem_entire_7} below). For this purpose, define 
\begin{align}\label{lem_entire_6}
	\chi^* := \inf\left \{ \chi \in \R, w^{\tilde{\chi}} \geq w ~ \text{for all }~ \tilde{\chi} \geq \chi \right \}.
\end{align}
Then, our goal can be obtained by proving that $\chi^* \leq 0$. By the previous argument, we already know that $\chi^* \leq 2A$. 
Since $w((z', - \infty), \tau) = \alpha_-$, it follows from \eqref{lem_entire_1} that $\chi^* > -\infty$, since otherwise we would have 
\begin{align*}
	\alpha_-
	=
	w^{-\infty}((z',\zn),\tau)
	\geq
	w,
\end{align*}
leading to a contradiction since $w((z', + \infty), \tau) = \alpha_+ > \alpha_-$. Thus, we conclude $- \infty < \chi^* \leq 2A$. 

Assume that $\chi^* > 0$, and define $E' := \{((z',\zn), \tau) \in \R^N \times \R; ~ |\zn| \leq A  \}$. If $\inf_{E'} (w^{\chi^*} - w) > 0$, then there exists $\delta_0 \in (0, \chi^*)$ such that $w \leq w^{\chi^* - \delta}$ in $E'$ for all $\delta \in (0, \delta_0)$. Since $w \leq w^{\chi^* - \delta}$ on $\partial E'$, we deduce from a similar argument as above that $w \leq w^{\chi^* - \delta}$ in $\{((z',\zn), \tau) \in \R^N \times \R; ~ |\zn| \geq A \}$. This contradicts the definition of $\chi^*$ in \eqref{lem_entire_6} so that $\inf_{E'} (w^{\chi^*} - w) = 0$. Thus, we must have a sequence $(({z}'_n, {z}_n), {t}_n)$ and $\tilde{z}_\infty \in [-A, A]$ such that 
\begin{align*}
	w(({z}'_n, {z}_n), {t}_n) 
	- w^{\chi^*}(({z}'_n, {z}_n), {t}_n) \rightarrow 0
	,~~
	{z}_n \rightarrow {z}_\infty 
	~\text{as}~ 
	n \rightarrow\infty.
\end{align*}
Define ${w}_n((z',\zn),\tau) := w((z' + {z}'_n, \zn), \tau + {t}_n)$. Since $w_n$ is bounded uniformly in $C^{2 + \theta, 1 + \theta/2}(\R^N \times \R)$ by Lemma \ref{lem_locregularity}, $w_n$ converges in $C^{2,1}_{loc}$ to a solution  $\twi$ of \eqref{eqn_entire}. Define $\tilde{Z}$ by
\begin{align*}
	\tilde{Z}((z',\zn),\tau) 
	:=
	[\vp(\twic) - \vp(\twi)]((z',\zn),\tau).
\end{align*}
Since $\vp$ is strictly increasing, by \eqref{lem_entire_6} we have
\begin{align}\label{lem_entire_11}
\begin{cases}
	\tilde{Z}((z',\zn),\tau)
	\geq 0
	~ \text{in}
	~ \R^N \times \R
	\\
	\tilde{Z}((0, {z}_\infty), 0) 
	= \lim_{n\rightarrow\infty} 
	[{\vp({w}^{\chi^*}_n)}
	-
	{\vp({w}_n)}]
	((0,{z}_n),0)
	=
	\lim_{n \rightarrow \infty} 
	[{\vp(w^{\chi^*})}
	-{\vp(w)}]
	(({z}'_n, {z}_n), {t}_n)
	= 0.
\end{cases}
\end{align}
Then, direct computation gives
\begin{align*}
	\tilde{Z}_\tau - \vp'(\twic) \Delta \tilde{Z}
	&=
	\vp'(\twic)[\Delta \vp(\twic) + f(\twic)]
	-
	\vp'(\twi)[\Delta \vp(\twi) + f(\twi)]
	\\
	&
	- \vp'(\twic) [\Delta \vp(\twic) - \Delta \vp(\twi)]
	\\
	&=
	[\vp'(\twic) - \vp'(\twi)]\Delta \vp(\twi)
	+[\vp'(\twic)f(\twic) - \vp'(\twi)f(\twi)],
\end{align*}
If $\tilde{Z} = 0$ we obtain $\tilde{Z}_\tau - \vp'(\twic) \Delta \tilde{Z} = 0$. If $\tilde{Z} > 0$, we obtain 
\begin{align*}
	\tilde{Z}_\tau - \vp'(\twic) \Delta \tilde{Z}
	&=
	\left(
	\dfrac{[\vp'(\twic) - \vp'(\twi)]\Delta \vp(\twi)+[\vp'(\twic)f(\twic) - \vp'(\twi)f(\twi)]}{\vp(\twic) - \vp(\twi)} 
	\right)\tilde{Z}	
	\\
	&\geq
	- C \tilde{Z},
\end{align*}
for some positive constant $C$, where the last inequality follows from \eqref{lem_entire_10} and the fact that $\Delta\vp(\twi)$ is uniformly bounded in the whole space.
Since by \eqref{lem_entire_11} $\tilde{Z}$ attains a non-positive minimum at $((0,{z}_\infty),0)$, we deduce from the maximum principle applied on the domain $\R^N \times (-\infty, 0]$ that  $\tilde{Z} = 0$  for all $(z',\zn) \in \R^N,~\tau \leq 0$.  Hence, $\tilde{Z} \equiv 0$ in $\R^N \times \R$. This implies that
\begin{align*}
	\twi((0,0),0) 
	= 
	\twi((\rho, \chi^*), \tilde{\tau})
	= 
	\twi((2\rho, 2\chi^*), 2\tilde{\tau})
	=
	\cdots
	=
	\twi((k\rho, k \chi^*), k \tilde{\tau})
\end{align*}
 for all $k \in \mathbb{Z}$, contradicting the fact that $\twi((k\rho, k \chi^*), k \tilde{\tau}) \rightarrow \alpha_+$ as $k \rightarrow  \infty$ and $\twi((k\rho, k \chi^*), k \tilde{\tau}) \rightarrow \alpha_-$ as $k \rightarrow - \infty$.

Thus, we have $\chi^* \leq 0$, and therefore
\begin{align}\label{lem_entire_7}
	w((z',\zn),\tau) \leq 
	w^0((z',\zn),\tau)
	= w((z' + \rho, \zn), \tau + \tilde{\tau})
\end{align}
holds for any $\rho \in \R^{N - 1}, \tilde{\tau} \in \R$.

 We now show that $w$ only depends on $\zn$. Suppose $w$ depends on $z'$ and $\tau$. Then, there exist $z'_1, z'_2 \in \R^{N - 1}, \zn \in \R$ and $t'_1, t'_2 \in \R$ such that 
\begin{align}\label{lem_entire_8}
	w((z'_1, \zn), t_1') < w((z'_2, \zn), t'_2).
\end{align}
 Then, by letting $z' = z_2',~\rho =  z_1' - z_2'$ and $\tau = t_2',~ \tilde{\tau} = t_1' - t_2'$ in the inequality \eqref{lem_entire_7}, we deduce
\begin{align*}
	w((z_2', \zn) ,t_2') \leq w((z_1', \zn), t_1'),
\end{align*}
 contradicting \eqref{lem_entire_8}. This implies that $w$ only depends on $\zn$, namely $w((z',\zn),\tau) = \psi(\zn)$. 
  Finally, from the definition of $\chi^*$, we have that $\psi$ is increasing. 
\end{proof}

{\bf Proof of Theorem \ref{thm_asymvali}.}
We first prove $(i)$. Recall that $d(x,t)$ is the 
{cut-off}
signed distance function to the interface $\Gamma_t$ moving according to equation \eqref{eqn_motioneqn}, and $d^\e(x,t)$ is the signed distance function corresponding to the interface
\begin{align*}
	\Gamma_t^\e := \{ x \in D, ~ u^\e(x,t) = \alpha \}.
\end{align*}

Let $T_1$ be an arbitrary constant such that $\frac{T}{2} < T_1 < T$. Assume by contradiction that \eqref{thm_asymvali_i} does not hold. Then, there exist $\eta > 0$ and sequences $\e_k \downarrow 0,~ t_k \in [\rho t^{\e_k}, T],~ x_k \in D$ such that $\alpha_+ - \eta > \alpha > \alpha_- + \eta$ and 
\begin{align}\label{eqn_asymproof_1}
	\left|
		u^{\e_k}(x_k, t_k) 
		- U_0
		\left(
			\dfrac{d^{\e_k}(x_k, t_k)}{\e_k}
		\right)
	\right|
	\geq \eta .
\end{align}
For the inequality \eqref{eqn_asymproof_1} to hold, by Theorem \ref{Thm_Propagation} and $U_0(\pm \infty) = \alpha_\pm$, we need
\begin{align*}
	d^{\e_k}(x_k, t_k) = \mathcal{O}(\e_k).
\end{align*}
With these observations, and also by Theorem \ref{Thm_Propagation}, there exists a positive constant $\tilde{C}$ such that 
\begin{align}\label{eqn_asymproof_2}
	| d(x_k,t_k) | \leq \tilde{C} \e_k
\end{align}
for $\e_k$ small enough.

 If $x_k \in \Gamma^{\e_k}_{t_k}$, then the left-hand side of \eqref{eqn_asymproof_1} vanishes, which contradicts this inequality. Since the sign can either be positive or negative, by extracting a subsequence if necessary we may assume that
\begin{align}\label{eqn_asymproof_7}
	u^{\e_k}(x_k,t_k) - \alpha > 0~~ \text{for all}~~ k \in \mathbb{N},
\end{align}
which is equivalent to 
\begin{align*}
	d^{\e_k}(x_k,t_k) > 0 ~~ \text{for all}~~ k \in \mathbb{N}.
\end{align*}

By \eqref{eqn_asymproof_2}, each $x_k$ has a unique orthogonal projection $p_k := p(x_k,t_k) \in \Gamma_{t_k}$. Let $y_k$ be a point on $\Gamma^{\e_k}_{t_k}$ that has the smallest distance from $x_k$, and therefore $u^{\e_k}(y_k,t_k) = \alpha$. Moreover, we have
\begin{align}\label{eqn_asymproof_3}
	u^{\e_k}(x, t_k) > \alpha
	~~ \text{if}
	~~ \Vert x - x_k \Vert < \Vert y_k - x_k \Vert.
\end{align}

We now rescale $u^{\e_k}$ around $(p_k,t_k)$.  Define 
\begin{align}\label{eqn_asymproof_10}
	w^k(z, \tau) 
	:=
	u^{\e_k}(p_k + \e_k \mathcal{R}_kz, t_k + \e^2_k \tau),
\end{align}
where $\mathcal{R}_k$ is a orthogonal matrix in $SO(N,\R)$ that rotates the $\zn$ axis, namely the vector $(0, \cdots ,0, 1) \in \R^N$  onto the unit normal vector to $\Gamma_{t_k}$ at $p_k \in \Gamma_{t_k}$, say $\dfrac{x_k - p_k}{\Vert x_k - p_k \Vert}$. To prove our result, we use Theorem \ref{Thm_Propagation} which gives information about $u^{\e_k}$ for $t_k + \e^2_k \tau \geq t^{\e_k}$.
 Then, since $\Gamma_t$ is separated from $\partial D$ by some positive distance, $w^k$ is well-defined at least on the box 
\begin{align*}
	B_k
	:=
	\left\{
		(z, \tau) \in \R^N \times \R
		:
		|z| \leq \dfrac{c}{\e_k}
		,~~
		- (\rho - 1) \dfrac{{|\ln \e_k|}}{\mu}
		\leq \tau
		\leq \dfrac{T - T_1}{\e_k^2}
	\right\},
\end{align*}
for some $c > 0$. We remark  that $B_k \subset B_{k + 1}, k \in \mathbb{N}$ and $\lim_{k \rightarrow \infty} B_k = \R^N \times \R$.  Writing $\mathcal{R}_k = (r_{ij})_{1 \leq i,j \leq N}$, we remark that $\mathcal{R}_k^{-1} = \mathcal{R}_k^T$, which implies that
\begin{align}\label{eqn_asymproof_9}
	\sum_{i = 1}^N r_{ \ell i }^2 = 1
	,~  
	\sum_{i = 1, j \neq m}^N r_{j i} r_{\ell i} = 0.
\end{align}
Since
\begin{align*}
	\partial_{z_i}^2 \vp(w^k)
	=
	\e_k^2 \sum_{j = 1}^N  \sum_{\ell = 1}^N r_{j i} r_{\ell i} \partial_{x_\ell x_j} \vp(u^{\e_k}),
\end{align*}
we have
\begin{align*}
	\Delta \vp(w^k)
	&=
	\e_k^2 \sum_{i = 1}^N \partial_{z_i}^2 \vp(w^k)
	\\
	&=
	\e_k^2 \sum_{i = 1}^N  \sum_{\ell = 1}^N r_{\ell i}^2 \partial_{x_\ell}^2 \vp(u^{\e_k})
	+
	\e_k^2 \sum_{i = 1}^N \sum_{j,\ell = 1, j \neq \ell}^N  r_{j i} r_{\ell i} \partial_{x_\ell x_j} \vp(u^{\e_k})
	\\
	&
	=\e_k^2 \Delta \vp(u^{\e_k}).
\end{align*}
Thus, we obtain
\begin{align*}
	w^k_\tau = \Delta \vp(w^k) + f(w^k)
	~~ \text{in}
	~~ B_k.
\end{align*}
From the propagation result in Theorem \ref{Thm_Propagation} and the fact that the rotation matrix $\mathcal{R}_k$ maps the $\zn$ axis to the unit normal vector of $\Gamma_t$ at $p_k$, there exists a constant $C > 0$ such that 
\begin{align}\label{eqn_asymproof_4}
	\zn \geq C \Rightarrow w^k(z,\tau) \geq \alpha_+ - \eta > \alpha
	,~~
	\zn \leq -C \Rightarrow w^k(z,\tau) \leq \alpha_- + \eta < \alpha
\end{align}
as long as $(z,\tau) \in B_k$.

It follows from the first line of \eqref{thm_propagation_1} that $\alpha_- - \eta_0 \leq w^k \leq \alpha_+ + \eta_0$ for $k$ large enough. Then, by Lemma \ref{lem_locregularity} we can find a subsequence of $(w^k)$ converging to some $w \in C^{2,1}(\R^N \times \R)$ which satisfies 
\begin{align*}
	w_\tau 
	= \Delta \vp(w) + f(w)
	~~ on
	~~ \R^N \times \R.
\end{align*}
From Remark \ref{rmk_thm13} we can deduce \eqref{lem_entire_1}.  Then, by Lemma \ref{lem_entire}, there exists $z^* \in \R$ such that 
\begin{align}\label{eqn_asymproof_5}
	w(z, \tau)
	=
	U_0(\zn - z^*).
\end{align}

Define sequences of points $\{ z_k \}, \{ \tilde{z}_{k} \}$ by
\begin{align}\label{eqn_asymproof_11}
	z_k := 
	\dfrac{1}{\e_k} \mathcal{R}^{-1}_k(x_k - p_k), ~~
	\tilde{z}_{k} := 
	\dfrac{1}{\e_k} \mathcal{R}^{-1}_k(y_k - p_k).
\end{align}
From \eqref{eqn_asymproof_2} and Theorem \ref{Thm_Propagation}, we have 
\begin{align*}
	&|d(x_k,t_k)|
	=
	\Vert x_k - p_k \Vert 
	= \mathcal{O}(\e_k)
	,~~\\
	&
	\Vert y_k - p_k \Vert 
	\leq 
	\Vert y_k - x_k \Vert 
	+ \Vert x_k - p_k \Vert 
	=
	|d^{\e_k}(x_k,t_k)|
	+
	|d(x_k,t_k)|
	= \mathcal{O}(\e_k)
\end{align*}
(see Figure \ref{fig1}), which implies that the sequences $z_k$ and $\tilde{z}_k$ are bounded. 
Thus, there exist subsequences of $\{ z_k \}, \{ \tilde{z}_k \}$ and $z_\infty, \tilde{z}_\infty \in \R^N$ such that
\begin{align*}
	z_{k_n} \rightarrow z_\infty
	,~~
	\tilde{z}_{k_n} \rightarrow \tilde{z}_\infty
	,~~
	\text{as}~~ k \rightarrow \infty.
\end{align*}

\begin{figure}[h]
\centering
\begin{subfigure}[b]{0.4\textwidth}
	\includegraphics[scale = 0.4]{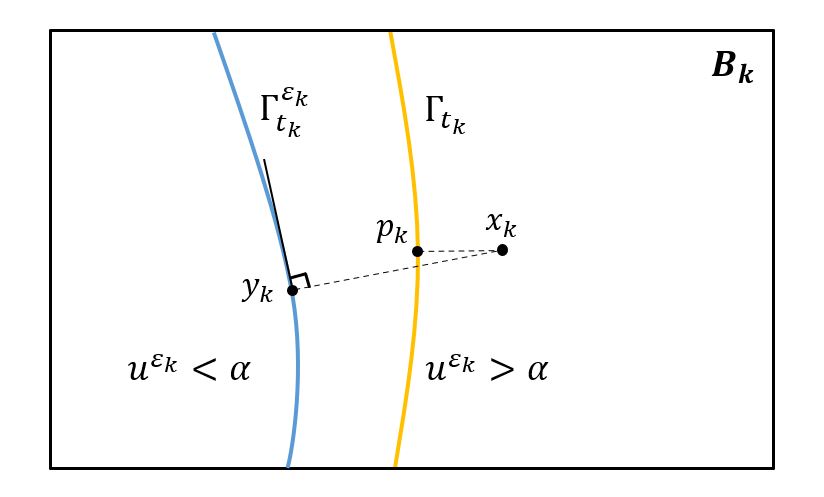}
	\caption{Points $x_k, y_k, p_k$ and interfaces $\Gamma_{t_k}, \Gamma_{t_k}^{\e_k}$ inside the box $B_k$.}
\end{subfigure}
\begin{subfigure}[b]{0.4\textwidth}
	\includegraphics[scale = 0.4]{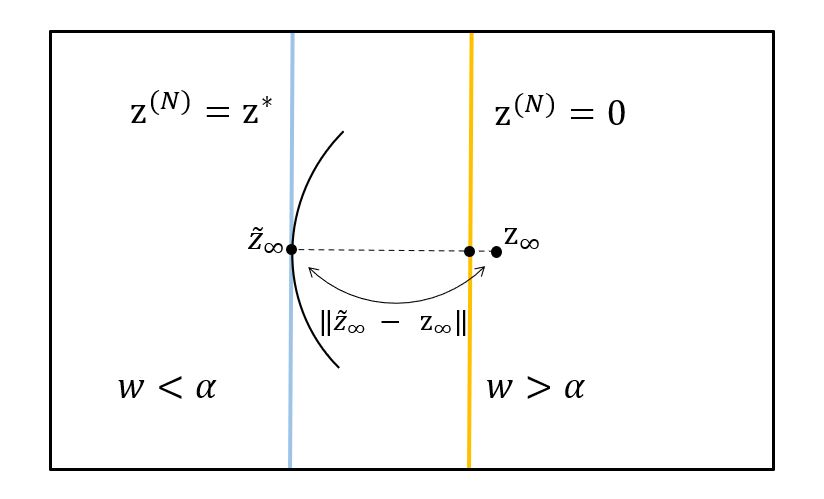}
	\caption{Points $z_\infty$ and $\tilde{z}_\infty$ and hyperplanes $\zn = z^*, \zn = \zn_\infty$.}
\end{subfigure}
\caption{In (a) the distance between $\Gamma_{t_k}$ and $\Gamma_{t_k}^{\e_k}$ is of $\mathcal{O}(\e_k)$. In (b), since we rescale space by $\e^{-1}$, the distance between two hyperplanes is of $\mathcal{O}(1)$.}\label{fig1}
\end{figure}

Since the normal vector to $\Gamma_{t_k}$ at $p_k$ is equal to $x_k - p_k$, and the mapping $\mathcal{R}_k^{-1}$ sends the unit normal vector to $\Gamma_{t_k}$ at $p_k$ to the vector $(0, \cdots 0, 1) \in \R^N$, we conclude $z_\infty$ must lie on the $\zn$ axis so that we can write
\begin{align*}
	z_\infty = (0, \cdots, 0, \zn_\infty).
\end{align*}
Since, by \eqref{eqn_asymproof_7},
\begin{align*}
	w(z_\infty, 0) 
	=
	\lim_{k_n \rightarrow \infty}
	w^{k_n}(z_{k_n}, 0)
	=
	\lim_{k_n \rightarrow \infty}
	u^{\e_{k_n}}(x_{k_n}, t_{k_n})
	\geq \alpha,
\end{align*}
we deduce from \eqref{eqn_asymproof_5} and the fact that $U_0$ is an increasing function that
\begin{align*}
	w(z_\infty, 0)
	= U_0(\zn_{\infty} - z^*) \geq \alpha
	\Rightarrow
	\zn_{\infty} \geq z^*.
\end{align*}

From the definition of $y_{k_n}$ and \eqref{eqn_asymproof_10}, we have 
\begin{align}\label{eqn_asymproof_6}
	w(\tilde{z}_\infty, 0) 
	= \lim_{k \rightarrow \infty}  w^{k_n}(\tilde{z}_{k_n}, 0)
	= \lim_{k \rightarrow \infty} u^{\e_{k_n}}(y_{k_n}, t_{k_n})
	= \alpha.
\end{align}

Next, we show that
\begin{align}\label{eqn_asymproof_8}
	w(z,0) \geq \alpha~ \text{if} ~
	\Vert z - z_\infty \Vert \leq \Vert \tilde{z}_\infty - z_\infty \Vert.
\end{align}
Choose $z \in \R^N$ satisfying $\Vert z - z_\infty \Vert \leq \Vert \tilde{z}_\infty - z_\infty \Vert$ and a sequence $a_{k_n} \in \R^+$ such that $a_{k_n} \rightarrow \Vert z - z_\infty \Vert$ and $\e_{k_n} a_{k_n} \leq \Vert x_{k_n} - y_{k_n} \Vert $ as $k \rightarrow \infty$. Then, we define sequences $n_{k_n}$ and $b_{k_n}$ by
\begin{align*}
	n_{k_n} = \dfrac{z - z_{k_n}}{\Vert z - z_{k_n} \Vert}
	,~
	b_{k_n} = a_{k_n} n_{k_n} + z_{k_n}.
\end{align*}
Note that $b_{k_n} \rightarrow z$ as $k \rightarrow \infty$. Then, by \eqref{eqn_asymproof_11}, we obtain
\begin{align*}
	w(z,0) 
	&= \lim_{k_n \rightarrow \infty} w^{k_n}(b_{k_n}, 0)=
	\lim_{k_n \rightarrow \infty} 
	u^{\e_{k_n}} ( p_{k_n} + \e_{k_n} \mathcal{R}_{k_n}(a_{k_n}  k_n + z_{k_n}),t_{k_n})
	\\
	&	
	= \lim_{k _n\rightarrow \infty} u^{\e_{k_n}}(\e_{k_n} a_{k_n} \mathcal{R}_{k_n} n_{k_n} + x_{k_n}, t_{k_n})
	\geq \alpha,
\end{align*}
where the last inequality holds by \eqref{eqn_asymproof_3}.

Note that \eqref{eqn_asymproof_5} implies $\{ w = \alpha \} = \{(z,\tau) \in \R^N \times \R, \zn = z^* \}$.
Thus, we have either $z_\infty = \tilde{z}_\infty$ or, in view of \eqref{eqn_asymproof_5} , \eqref{eqn_asymproof_6} and \eqref{eqn_asymproof_8}, that the ball of radius $|| \tilde{z}_\infty - z_\infty ||$ centered at $z_\infty$ is tangent to the hyperplane $\zn = z^*$ at $\tilde{z}_\infty$. Hence, 
$\tilde{z}_\infty$ is a point on $\zn$ axis. With this observation and \eqref{eqn_asymproof_5}, we have
\begin{align*}
	\tilde{z}_\infty
	= (0, \cdots, 0, z^*).
\end{align*}
This last property implies
\begin{align}\label{eqn_asymproof_12}
	\dfrac{d^{\e_{k_n}}(x_{k_n},t_{k_n})}{\e_{k_n}} 
	= \dfrac{\Vert x_{k_n} - y_{k_n} \Vert}{\e_{k_n}}
	= \Vert \mathcal{R}_{k_n} \left( z_{k_n} - \tilde{z}_{k_n} \right) \Vert
	= \Vert z_{k_n} - \tilde{z}_{k_n} \Vert
	\rightarrow 
	\Vert  z_\infty - \tilde{z}_\infty \Vert
	= \zn_\infty - z^*.
\end{align}
We have therefore reached a contradiction since, by \eqref{eqn_asymproof_2}, \eqref{eqn_asymproof_5} and \eqref{eqn_asymproof_12},
\begin{align*}
	0 
	&=
	\vert w(z_\infty, 0) - U_0 (\zn_\infty - z^*) \vert
	\\
	&= 
	\left\vert \lim_{k_n \rightarrow \infty}
	\left[
	 w^{k_n}(z_{k_n}, 0) 
	- U_0 
	\left(
		\dfrac{d^{\e_{k_n}}(x_{k_n},t_{k_n})}{\e_{k_n}}
	\right)
	\right]
	\right\vert
	\\
	&= 
	\left\vert \lim_{k_n \rightarrow \infty}
	\left[
	 u^{\e_{k_n}}(x_{k_n},t_{k_n}) 
	- U_0 
	\left(
		\dfrac{d^{\e_{k_n}}(x_{k_n},t_{k_n})}{\e_{k_n}}
	\right)
	\right]
	\right\vert,
\end{align*} 
contradicting \eqref{eqn_asymproof_1}.

For the proof of $(ii)$, we use the same method as in \cite{MH2012}. \hfill \qed

\section*{Appendix: Mobility and surface tension}

Mobility is defined as a linear response of the speed of traveling wave to the 
external force.  More precisely, motivated by (4.1) and (4.2) in \cite{Spohn1993}, let us consider
the nonlinear Allen-Cahn equation with external force $\delta$ on $\R$
for small enough $|\delta|$:
\begin{equation}  \label{eq:AC-delta}
u_t = \varphi(u)_{zz} +f(u)+\delta, \quad z\in \R,
\end{equation}
and corresponding traveling wave solution $U=U_\delta(z)$ with speed $c(\delta)$:
\begin{align}  \label{eq:TW-delta}
	& \vp(U_\delta)_{zz} +c(\delta) U_{\delta z}+f(U_\delta)+\delta=0, ~~ z\in \R,  \\
	& U_\delta(\pm\infty)= \alpha_{\pm, \delta},
\notag
\end{align}
where $\alpha_{\pm,\delta}$ are two stable solutions of $f(u)+\delta=0$.
Then, we define the mobility by 
$$
\mu_{AC}:=-\frac{c'(0)}{ \a_+- \a_-},
$$
with a normalization factor $\a_+-\a_-$ as in \cite{Spohn1993}; compare (4.6) and (4.7) in 
\cite{Spohn1993} noting that the boundary conditions at $\pm\infty$ are switched so that
we have a negative sign for $\mu_{AC}$.

To derive a formula for $\mu_{AC}$, we multiply $\fa(U_\de)_z$ to 
\eqref{eq:TW-delta} and integrate it over $\R$ to obtain
\begin{align}  \label{eq:94}
c(\de) \int_\R U_{\de z} \fa(U_\de)_z dz + \de (\fa(\a_+)-\fa(\a_-)) =O(\de^2),
\end{align}
by noting that
\begin{align*}
& \int_\R \fa(U_{\de})_{zz} \fa(U_\de)_z dz = \frac12 \int_\R 
\big\{\big( \fa(U_{\de})_{z} \big)^2 \big\}_z dz =0, \\
& \int_\R \fa(U_{\de})_{z} dz = \fa(\a_{+,\de})-\fa(\a_{-,\de}) = \fa(\a_+)-\fa(\a_-)+O(\de), \\
&  \int_\R f(U_{\de}) \fa(U_\de)_z dz =  \int_\R f(U_{\de}) \fa'(U_\de) U_{\de z} dz 
= - \int_{\a_{-,\de}}^{\a_{+,\de}} W'(u)du = O(\de^2). 
\end{align*}
The last line follows by the change of variable $u=U_\de(z)$, $W'(u) = -f(u)\fa'(u)$
(recall  \eqref{eq:28}),
$\int_{\a_{-}}^{\a_{+}} W'(u)du=0$ and $W'(\a_\pm)=0$, $W'\in C^1$.
However, since one can at least formally expect 
$U_\de = U_0+O(\de)$ (recall \eqref{eqn_AsymptExp_U0} for $U_0$), by \eqref{intrinsic},
\begin{align*}
\int_\R U_{\de z} \fa(U_\de)_z dz
& = \int_\R U_{0 z} \fa(U_0)_z dz + O(\de) \\
& = \int_\R U_{0 z} \sqrt{2W(U_0(z))} dz + O(\de) \\
& = \int_{\a_-}^{\a_+} \sqrt{2W(u)} du + O(\de),
\end{align*}
by the change of variable $u=U_0(z)$ again.  This combined with \eqref{eq:94} leads to
\begin{align*}
c'(0) = - \frac{\fa(\a_+)-\fa(\a_-)}
{\int_{\a_-}^{\a_+} \sqrt{2W(u)} du}.
\end{align*}
Thus, the mobility is given by the formula
\begin{align}  \label{mobility}
\mu_{AC} =  \frac{ \fa_\pm^*}{\int_{\a_-}^{\a_+} \sqrt{2W(u)} du}
= \frac{ \fa_\pm^*}{\int_\R \fa'(U_0) U_{0z}^2(z)dz},
\end{align}
where
$$
\fa_\pm^* = \frac{\fa(\a_+)-\fa(\a_-)}{\a_+ - \a_-}.
$$

On the other hand, surface tension is defined as a gap of the energy
of the microscopic transition surface from $\a_-$ to $\a_+$ in normal
direction and that of the constant profile $\a_-$ or $\a_+$.  More precisely,
define the energy of a profile $u= \{u(z)\}_{z\in \R}$ by
$$
\mathcal{E}(u) = \int_\R \Big\{ \frac12\big(\fa(u)_z\big)^2 +W(u)\Big\} dz.
$$
Recall that the potential $W$ is defined by \eqref{eq:28}, and $W\ge 0$
and $W(\a_\pm)=0$ hold.  In particular, $W$ is normalized as $\min_{u\in \R}W(u)=0$
so that $\min_{u =u(\cdot)}\mathcal{E}(u)=0$.
Then, the surface tension is defined as
$$
\si_{AC} := \frac1{\fa_\pm^*} \min_{u: u(\pm \infty)=\a_\pm} \mathcal{E}(u),
$$
by normalizing the energy by $\fa_\pm^*$.  We observe 
 that $\mathcal{E}$ is defined through $\fa$.

Note that the nonlinear Allen-Cahn equation, that is \eqref{eq:AC-delta}
with $\de=0$, is a distorted gradient flow associated with $\mathcal{E}(u)$:
$$
u_t = - \frac{\de \mathcal{E}(u)}{\de\fa(u)}, \quad z \in \R,
$$
where the right hand side is defined as a functional derivative of
$\mathcal{E}(u)$ in $\fa(u)$, which is given by
$$
\frac{\de \mathcal{E}(u)}{\de\fa(u)} = - \fa(u)_{zz} -f(u(z)).
$$
Indeed, to see the second term $-f(u(z))$, setting $v=\fa(u)$, one can rewrite
$W(u) = W(\fa^{-1}(v))$ as a function of $v$ so that
\begin{align*}
\big(W(\fa^{-1}(v))\big)' & = W'(\fa^{-1}(v)) \big(\fa^{-1}(v)\big)'  \\
& = - f(\fa^{-1}(v)) \fa'(\fa^{-1}(v)) \frac1{\fa'(\fa^{-1}(v))}  \\
& = - f(\fa^{-1}(v)) = -f(u).
\end{align*}

We call the flow ``distorted'', since the functional derivative is taken in $\fa(u)$ 
and not in $u$.  One can rephrase this in terms of the change of
variables $v(z) = \fa(u(z))$.  Indeed, we have $\mathcal{E}(u)=
\widetilde{\mathcal{E}}(v)$ under this change, where
$$
\widetilde{\mathcal{E}}(v) = \int_\R \Big\{ \frac12 v_z^2 + W(\fa^{-1}(v))\Big\}dz,
$$
and
$$
\frac{\de \widetilde{\mathcal{E}}}{\de v}
= -v_{zz} - f(\fa^{-1}(v)) = -v_{zz} -g(v).
$$
Therefore, in the variable $v(z)$, the nonlinear Allen-Cahn equation can be rewritten as
\begin{align*}
v_t = \fa'(u) u_t = - \fa'(\fa^{-1}(v)) \cdot \frac{\de \widetilde{\mathcal{E}}}{\de v}
= \fa'(\fa^{-1}(v)) \big\{ v_{zz} +g(v)\}.
\end{align*}
This type of distorted equation for $v$ is sometimes called Onsager 
equation; see \cite{Mi}.

Now we come back to the computation of the surface tension $\si_{AC}$.
In fact, it is given by
\begin{align}  \label{eq:ST}
\si_{AC}= \frac1{\fa_\pm^*}  \int_\R V_{0z}^2 dz = 
\frac1{\fa_\pm^*}
\int_{\a_-}^{\a_+} \fa'(u) \sqrt{2W(u)} du,
\end{align}
where $V_0=\fa(U_0)$ and satisfies \eqref{eqn_AsymptExp_V0}.
Indeed, the second equality follows from \eqref{eq:29}.  To see the first equality,
by definition,
\begin{align*}
\si_{AC} =  \frac1{\fa_\pm^*} \min_{u: u(\pm \infty)=\a_\pm} \mathcal{E}(u)
=  \frac1{\fa_\pm^*}
\min_{v: v(\pm \infty)= \fa(\a_\pm)} \widetilde{\mathcal{E}}(v)
\end{align*}
and the minimizers of $\widetilde{\mathcal{E}}$ under the condition
$v(\pm \infty)= \fa(\a_\pm)$ are given by $V_0$ and its spatial shifts.  Thus,
$$
\si_{AC} =  \frac1{\fa_\pm^*} \widetilde{\mathcal{E}}(V_0) 
= \frac1{\fa_\pm^*} \int_\R \Big\{\frac12 V_{0z}^2 + W(\fa^{-1}(V_0)) \Big\} dz.
$$
However, since $V_{0z}= \sqrt{2W(U_0(z))}$ by \eqref{intrinsic},
we have $\int_\R W(\fa^{-1}(V_0)) dz = \int_\R \frac12 V_{0z}^2 dz$.
In particular, this implies the first equality of \eqref{eq:ST}.

By \eqref{second lambda_0} combined with \eqref{mobility} and
\eqref{eq:ST}, we see that $\la_0= \mu_{AC}\si_{AC}$.

\begin{rmk}
The linear case  $\fa(u) = \frak{K} u$ is discussed by Spohn \cite{Spohn1993},
in which $\frak{K}$ is denoted by $\kappa$.  In this case, since $\fa'=\frak{K}$
and $\fa_\pm^*=\frak{K}$, by \eqref{mobility} and \eqref{eq:ST}, we have
$\mu_{AC} = \big[\int_\R U_{0z}^2 dz\big]^{-1}$ and $\si_{AC} = \frak{K}\int_\R U_{0z}^2 dz$.
These formulas coincide with (4.7) and (4.8) in \cite{Spohn1993} by noting that
$U_0$ is the same as $w$ in \cite{Spohn1993} in the linear case except that
the direction is switched due to the choice of the boundary conditions.
\end{rmk}

\section*{Acknowledgement}
The authors are grateful to the professors Henri Berestycki and Francois Hamel for a useful discussion. 
P. El Kettani, D. Hilhorst and H.J. Park thank IRN ReaDiNet as well as the French-Korean project STAR.
 T. Funaki was supported in part by JSPS KAKENHI, Grant-in-Aid for Scientific Researches (A) 18H03672 and (S) 16H06338, and thanks also IRN ReaDiNet.
S. Sethuraman was supported by grant ARO W911NF-181-0311, a Simons Foundation Sabbatical grant, and by a JSPS Fellowship, and thanks Waseda U. for the kind hospitality during a sabbatical visit.

\end{document}